\documentclass[11pt]{amsart}   
\usepackage[margin=3.1cm]{geometry} 
\def\FINAL{1}
\def\ACKN{1}
\usepackage{newton}
\let\ge\geqslant
\let\le\leqslant
\begin{document}

\title{Models of curves over DVRs}
\author{Tim Dokchitser}
\thanks{This research is supported by the \hbox{EPSRC} grant EP/M016838/1}
\date{\today}
\begin{abstract}
Let $C$ be a smooth projective curve over a discretely valued field $K$, 
defined by an affine equation $f(x,y)=0$.
We construct a model of $C$ over the ring of integers of $K$ using a toroidal embedding 
associated to the Newton polygon of $f$. We show that under `generic' conditions
it is regular with normal crossings, and determine when it is minimal, the global sections 
of its relative dualising sheaf, and the tame part of the first \'etale cohomology of $C$. 
\end{abstract}

\maketitle

\section{Introduction}

The purpose of this paper is to construct regular models and 
study invariants of arithmetic surfaces using a toric resolution derived
from the defining equation.
Let $K$ be a field with a discrete valuation $\vK$ and residue field~$k$,
and $C/K$ a smooth projective curve 
specified by an affine equation $f(x,y)=0$.
We show that the Newton polytope $\Dv$ of $f$ with respect to~$\vK$,
under a `generic' condition called $\Dv$-regularity, determines explicitly

\begin{itemize}
\item[---]
the minimal regular normal crossings model of $C$ over $O_K$;
\item[---]
whether $C$ and $\Jac C$ have good, semistable and tame reduction;
\item[---]
the action $\Gal(\bar K/K)\acts\H(C_{\bar K},\Q_l)$ when $C$ is tamely ramified, 
and the action on its wild inertia invariants in general, for $l\ne\vchar k$;
\item[---]
a basis of global sections of the relative dualising sheaf;
\item[---]
the reduction map on points from the generic to the special fibre.
\end{itemize}
A regular model is usually constructed by starting with any model over $O_K$, and
repeatedly blowing up at points and components of the special fibre and taking normalisations. 
In effect, the toric resolution replaces repeated blow-ups along coordinate axes, and for 
$\Dv$-regular curves one such resolution is enough to get a good model. 

\subsection{Regular model}
\label{ssintromain}

Our main objects are the Newton polytopes of the defining equation $f = \sum a_{ij}x^iy^j$ 
of the curve $C$ over $K$ and over $O_K$,
$$
\begin{array}{lclrlll}
  \Delta & = & \text{convex hull}\bigl(\>  (i,j) &\bigm|\>a_{ij}\ne 0\>\bigr)  &\subset\R^2,\\[2pt]
  \Dv    & = & \text{lower convex hull}\bigl(\>  (i,j,\vK(a_{ij}))\!\!\! &\bigm|\>a_{ij}\ne 0\>\bigr)  &\subset\R^2\times\R.
\end{array}
$$
(The lower convex hull is those points $Q$ in the convex hull such that $Q-(0,0,\epsilon)$ is not
in it for any $\epsilon>0$.) Thus, $\Delta$ is a 2-dimensional convex polygon, and above every
point $P\in\Delta$ there is a unique point $(P,v(P))\in\Delta_v$. 
The function $v: \Delta\to\R$ is piecewise affine,
and breaks $\Delta$ into 
2-dimensional \emph{$v$-faces} and 1-dimensional \emph{$v$-edges},
the images of faces and edges of the polytope $\Dv$ under the homeomorphic projection to $\Delta$.
We will associate
\begin{itemize}
\item 
denominators $\delta_L$, $\delta_F\in\N$ to every $v$-edge $L$ and $v$-face $F$ (see \ref{notintden}),
\item 
proper 1-dimensional schemes $\bar X_F/k$ to every $v$-face $F$, and\\
finite 0-dimensional schemes $X_L/k$ to every $v$-edge $L$ (see \ref{defres}-\ref{defXFXL}),
\item
slopes $[s_L^1,s_L^2]$ at every $v$-edge $L$ from adjacent $v$-faces (see \ref{defslopes}).
\item
$\Delta(\Z)=\text{interior}(\Delta)\cap\Z^2$, the set of interior integer points.
\end{itemize}

\noindent
We say that $f$ (or $C$) is \emph{$\Delta_v$-regular} essentially%
\footnote{The condition for outer $v$-faces can be made slightly weaker, see \ref{dvreg}.}
\marginpar{Is that true? Check footnote}
if all $\bar X_F, X_L$ are smooth. Under these conditions\marginpar{need ref},
$C$ is the completion of $\{f=0\}\subset\G_m^2$ in the toric variety with fan $\Delta$;
see \S\ref{sBaker}. In particular, the genus of $C$ is $|\Delta(\Z)|$ and
it has a basis of regular differentials
{\scalebox{0.95}{$x^iy^j\frac{dx}{xyf'_y}$}}
indexed by $(i,j)\in\Delta(\Z)$.

As an example, let $\pi\in K$ be a uniformiser, and take a curve (left)
$$
C\colon y^2\!+\!\pi x^3 y\!+\!x^3\!+\!\pi^5=0
\qquad\quad 
\pbox[c]{10cm}{\INTROEXCHART}
\qquad\quad 
\pbox[c]{10cm}{\INTROEXMODEL}
$$
Then $\Delta$ is a quadrangle, with values of $v$ on $\Delta\cap\Z^2$ as above (middle).
It has $v$-faces $F_1$, $F_2$ with $\delta_{F_1}=6$, $\delta_{F_2}=3$, 
\marginpar{and list the slopes, as a displayed formula}
and five $v$-edges (4 outer, 1 inner). 
Here $C$ has genus 2 and is $\Dv$-regular, for any ground field $K$.


\marginpar{Instead say that for any curve we construct a model? And then prove regularity
if $C$ is $\Dv$-regular?} 

In \S\ref{sProof} we construct a proper flat model $\cC_\Delta/O_K$ of $C/K$, using a 
toroidal embedding associated to $\Dv$. For $\Dv$-regular curves the main result is

%

\begin{theorem}[\ref{mainthm1}+\ref{mainthmdiff}]
\label{imainthm1}
If $f$ is $\Dv$-regular, then $\cC_\Delta/O_K$ is a regular model of $C/K$ with strict normal crossings.
Its special fibre $\cC_k/k$ consists~of:

\begin{enumerate}
\item
A complete smooth curve $\bar X_F/k$ of multiplicity $\delta_F$, for every $v$-face~$F$.
\item
For every $v$-edge $L$ with slopes $[s_1^L,s_2^L]$ 
pick $\frac{m_i}{d_i}\in\Q$ so that
$$
  s_1^L=\frac{m_0}{d_0}\!>\!\frac{m_1}{d_1}\!>\!\ldots\!>\!\frac{m_r}{d_r}\!>\!\frac{m_{r+1}}{d_{r+1}}=s_2^L
    \>\>\quad\text{with}\>\>\>
  \scalebox{0.9}{$\left|
  \begin{matrix}
  m_i\!\!\! & m_{i+1} \cr
  d_i\!\!\! & d_{i+1} \cr
  \end{matrix}
  \right|$}=1.
$$
Then $L$ gives $|X_L(\bar k)|$ chains of $\P^1$s of length $r$ from%
\footnote{If $L$ is at the boundary of $\Delta$, it has only one adjacent $v$-face
  and there is no $\bar X_{F_2}$.}
$\bar X_{F_1}$ to $\bar X_{F_2}$ 
and multiplicities $\delta_L d_1,...,\delta_L d_r$.%
\footnote{If $r=0$, this is interpreted as $\bar X_{F_1}$ meeting $\bar X_{F_2}$ transversally 
at $|X_L(\bar k)|$ points in the inner case, and no contribution from $L$ in the outer case.
In the outer case, this happens precisely when $\delta_L=\delta_{F_1}$.}
The absolute Galois group $\Gal(k^s/k)$ permutes the chains through its action on $X_L(\bar k)$.
\end{enumerate}
If $f'_y$ is not identically zero, then the differentials 
$$
  \pi^{\lfloor v(i,j)\rfloor} \omegaij, \qquad (i,j)\in\Delta(\Z)
$$ 
form an $O_K$-basis of global sections of the relative dualising sheaf
$\omega_{\scriptscriptstyle \cC_\Delta/O_K}$.
\end{theorem}

In the example above, the special fibre is shown on the right.
See Theorem \ref{mainthm2} for a slightly more general statement, for arbitrary curves.

\subsection{Reduction}
\label{ssIntroRed}

There is a natural description, in terms of $\Delta_v$,
of the dual graph of the special fibre (\ref{dualgraph}), 
of the minimal regular model with normal crossings (\ref{minmain})
and of the reduction map of points to the special fibre (\ref{redpts}).
From these one can also deduce the following criterion for good, semistable, and tame%
\footnote{By `tame' we mean semistable over some tamely ramified extension 
of $K$; this is automatically the case if $\vchar k=0$ or $\vchar k>2\genus(C)+1$}
reduction for $\Dv$-regular curves and their Jacobians.

Write $\Delta(\Z)^F\subseteq\Delta(\Z)$ for points that are in the interiors of $v$-faces,
and $\Delta(\Z)^L$ for the others (lying on $v$-edges). We let subscripts,
such as $\Delta(\Z)_\Z$ or $\Delta(\Z)^F_{\Z_p}$, indicate that we further restrict to points 
with $v(P)\in\Z$ or $v(P)\in\Z_p$, respectively. (When $p=0$, the latter is interpreted 
as an empty condition). See \S\ref{ssnotation} for other notation.

\begin{theorem}[=\ref{redcond}]
Suppose $C/K$ is $\Dv$-regular, of genus$\,\ge\!1$. Then

\marginpar{$F(\Z)$ undefined yet; not principal $v$-faces}

\begin{itemize}
\item[(1)]
$C$ is $\Dv$-regular over every finite tame extension $K'/K$.
\item[(2)]
$C$ has good reduction $\iff$ $\Delta(\Z)\!=\!F(\Z)$ for some $v$-face $F$ with $\delta_F\!=\!1$.
\item[(3)]
$C$ is semistable $\iff$ every principal $v$-face $F$ (see \ref{defprincipal}) has $\delta_F=1$.
\item[(4)]
$C$ is tame $\iff$ every principal $v$-face $F$ has $\vchar k\nmid\delta_F$.\\
In this case, $C$
is $\Dv$-regular over every finite extension $K'/K$.
\end{itemize}
Suppose $k$ is perfect, and let $J$ be the Jacobian of $C$. Then
\begin{itemize}[resume]
\item[(5)]
$J$ has good reduction  $\iff$  $\Delta(\Z)=\Delta(\Z)^F_{\Z}$.
\item[(6)]
$J$ is semistable  $\iff$  $\Delta(\Z)=\Delta(\Z)_{\Z}$.
\item[(7)]
$J$ is tame  $\iff$ $\Delta(\Z)=\Delta(\Z)_{\Z_p}$, where $p=\vchar k$. In this case,\\
$J$ has potentially good reduction  $\iff$  $\Delta(\Z)^L=\emptyset$, and\\
$J$ has potentially totally toric reduction  $\iff$  $\Delta(\Z)^F=\emptyset$.
\end{itemize}
\end{theorem}

\subsection{\'Etale cohomology}
\label{ssIntroEt}

One of the fundamental invariants of a curve $C$ over a global field is its global Galois representation
and the associated $L$-function $L(C,s)$. 
To study it in practice, one needs to understand the corresponding local Galois representations.
They can be quite intricate when $C$ has bad reduction, but for tame $\Dv$-regular curves we can give
a complete description, as follows.

%
%
%
%
%

Let $K$ be complete with perfect residue field~$k$ of characteristic $p\ge 0$. 
Assume that $C/K$ is $\Dv$-regular and fix $l\ne p$. The representation in question is the
\'etale cohomology of $C$ with the action of the absolute Galois group,
$$
  G_K=\Gal(K^s/K) \quad\acts\quad H^1(C) = \H(C_{\bar K},\Q_l).
$$
Let $\Iwild\normal I_K\normal G_K$ be the wild inertia and the inertia subgroups.
It turns out that the action of $I_K$ on the subspace of wild inertia invariants 
$H^1(C)^{\Iwild=1}$ (which is all of $H^1(C)$ if $C$ is tame) 
depends only on $\Dv$, and has the following elementary description.

Fix a uniformiser $\pi$ of $K$. 
Each point $P\in \Delta(\Z)_{\Z_p}$ defines a tame 
character $\chiP: I_K\to\text{\{roots of unity\}}$ by $\sigma\mapsto \sigma(\pi^{v(P)})/\pi^{v(P)}$.
Let $V^{\ab}_{\tame}$, $V^{\toric}_{\tame}$ be the unique continuous representations of $I_K$
over $\Q_l$ that decompose over $\bar\Q_l$ as
$$
  V_{\tame}^{\ab}\>\> \iso_{\scriptscriptstyle\bar\Q_l}\!\!
  \bigoplus_{P\in \Delta(\Z)^F_{\Z_p}} \!\!\!\! (\chiP\oplus\chiPinv), \qquad
  V_{\tame}^{\toric}\>\> \iso_{\scriptscriptstyle\bar\Q_l}\!\!
  \bigoplus_{P\in \Delta(\Z)^L_{\Z_p}} \!\!\!\! \chiP.
$$
%
%
Then there is an isomorphism of $I_K$-modules (Theorem \ref{tamecoh}),
%
\marginpar{$\Sp_2$ needs to be defined. Note that for $V_l$ it may be only used 
  for inertia, otherwise twist it by 1}
$$
  H^1(C)^{\Iwild=1}\>\>\iso\>\> 
  V_{\tame}^{\ab} \>\>\oplus\>\> V_{\tame}^{\toric}\!\tensor\!\Sp_2.
$$
In particular, the dimension of the wild inertia invariants is 
$2|\Delta(\Z)_{\Z_p}|$.

When $k$ is finite,
we describe $H^1(C)^{\Iwild=1}$ as a full $G_K$-represen\-tation.
To a $v$-face $F$ of $\Delta$ we associate a scheme $\bar X_F^{\tame}/O_K$,
that depends only on the coefficients $a_P$ for $P$ along $F$. It is tame,
defined by an equation that has only one $v$-face, and its $H^1$ admits an 
explicit description in terms of point counting (Theorem \ref{onetame}).
Similarly we get schemes $X_L^{\tame}/O_K$ for $v$-edges $L$.
Then there is in isomorphism of $G_K$-representations (see Theorem \ref{tamedec})
$$
  H^1(C)^{\Iwild=1}\ominus\triv\>\>\>\iso\>\> 
    \bigoplus\subsmalltext4{$F$}{$v$-faces} (H^1(\bar X_F^{\tame})\ominus\triv)
      \>\>\>\oplus\>
    \bigoplus\subsmalltext4{$L$ inner}{$v$-edges} (H^0(X_L^{\tame})\tensor\Sp_2).
$$

\subsection{Applications}

As an illustration, we
\marginpar{could mention semistable hyperelliptic curves in M2D2}

\begin{itemize}
\item[---]
recover Tate's algorithm for elliptic curves (\S\ref{sEll}),
\item[---]
determine regular models of Fermat curves $x^p+y^p=1$ over $\Z_p$ and tame extensions
  of $\Z_p$, that seem to be unavailable (\S\ref{sFermat}),
\item[---]
compute regular models and their differentials for some curves from the literature, in an
elementary way (\ref{minremark}, \ref{expss}, \ref{ex188}).
\end{itemize}

Generally, on the computational side our motivation came from determining arithmetic invariants
of curves that enter the Birch--Swinnerton-Dyer conjecture, such as the differentials, conductor
and Tamagawa numbers. 
On the theoretical side, for $\Dv$-regular curves we also recover explicit versions of various
classical results: existence of (minimal) model with normal crossings,
Saito's criterion for wild reduction (\cite[Thm 3]{Saito}; see \cite{Hal} as well),
and Kisin's theorem on the continuity of $l$-adic representations for $H^1$ of curves
in the tame case (see \ref{remkisin}). It also extends some computations of regular models as 
tame quotients \cite{Vie,LorD,CES} to the wild case.
%
%
%

\subsection{Remarks}

(a) Note that Theorem \ref{imainthm1} applies to curves with wild reduction, because the construction 
is `from the bottom' (toroidal embedding) rather than `from the top' (taking a quotient of a semistable
model over some larger field). 
In fact, $\Dv$-regularity is automatic for the most `irregular' reduction types,
the ones where the only points in $\Dv\cap\Z^3$ are vertices of~$\Dv$.
(This is the case with the example above.)
Then the shape of the special fibre of their regular models does not depend 
on the residue characteristic at all. 
\marginpar{$y^5 + p x^2 y^3 + p^2 x^5 y = x^{11} + p^{13}$ (genus 20)}
This explains why for example elliptic curves 
$y^2=x^3+\pi$, $y^2=x^3+\pi^5$,  
$y^2=x^3+\pi x$, $y^2=x^3+\pi^3 x$
are always of Kodaira type II, II*, III, III*, irrespectively of whether $\vchar k$ is 2, 3 or $\ge 5$.

(b) Also, the theorem produces a regular model with normal crossings, which is somewhat more
pleasant than the minimal regular model for classification and computational purposes
(as it can be encoded combinatorially in the dual graph of the special fibre),
and sometimes for theoretical purposes as well; see e.g. \cite{Saito, Hal}.
%

(c) The technique appears to be useful for arithmetic schemes of higher dimension
as well, and over more general local rings. However, we only treat the case of curves over DVRs here, 
as it seems sufficiently interesting and already sufficiently involved.

(d) Although we have chosen to formulate the results for global equations, 
the construction of the regular model is really a local process (in effect 
a version of a repeated blow up). 

(e) In a way, our results for arithmetic surfaces are a natural extension of two classical 
lower-dimensional analogues --- 1-dimensional schemes over a field, and over a DVR.
For a plane curve over a field, Baker's theory computes the genus, points at infinity and differentials 
from a Newton polygon of the defining equation; we review this in \S\ref{sBaker}.
And for univariate $f(x)$ with coefficients in a complete DVR, the Newton polygon of $f$ 
reflects how $f$ factors, in effect forcing a decomposition of $\H[0]$ of the scheme $f=0$;
the results from \S\ref{ssIntroEt} simply extend this to bivariate $f(x,y)$ and their $\H[1]$.

(f) 
Finally, there is an obvious question as to what extent the theory extends to arbitrary curves, 
not just the $\Dv$-regular ones. 
\marginpar{If yes, one could possibly reprove
various fundamental results such as the semistable reduction theorem, existence of 
regular models, and the theorem of Saito.}
It appears~that~it does, but this requires additional work. 
We postpone this to a future paper.

\comment
Remark on 2IV-3. If principal components are permuted by inertia,
there is no locally $\Dv$-regular model. But could upgrade the theorem:
let $C$ be a complete intersection,
$$
  f(x,y,z)=0, \qquad \pi=z^n g(x,y,z), \qquad (n\!>\!0, z\nmid g).
$$
Then $C$ has a proper model obtained from the toroidal embedding of $f$.
It is describable, independent of the presentation of $f$ and regular at the 
components that correspond to faces away from the cone of $g$
(or $z^ng$, perhaps). 
This allows to resolve types like 2IV-3 recursively, and together with 
the normalisation of positive genus components handles general curves.
\endcomment


\newpage

\marginpar{\cite{LiuM} e.g. table on p.160 in tame cases.}

\marginpar{Cf. \cite[9/1.7]{Liu} for normal crossings (NC) models and 
  \cite[9/3.36]{Liu} for existence of minimal NCs.}

\marginpar{As we allow nodes, these are not SNCs?}

\marginpar{E.g. could use to study local rings such as $x^m y^n=\pi^r$ in tame extensions.}

\marginpar{It is an interesting question to extend the results 
  on cohomology and differentials to glued regular models as well.}


\marginpar{Numerology is well-known, see Remark \ref{sternbrocot}.}

\marginpar{Reflects different behaviour for genus 1 deficient curves.}

\subsection{Notation}
\label{ssnotation}

Throughout the paper we use the following notation:

\medskip

\marginpar{$f'_y\ne 0$ means $K(C)/K(x)$ is separable}

\begin{tabular}{llllll}
$K$, $\vK$ & field with a discrete valuation; in \S\ref{sBaker}, $K$ is any field \cr
$O_K$, $k$, $p$ & ring of integers, residue field, $\vchar k$ \cr
$\pi$ & a fixed choice of a uniformiser of $K$\cr
$f$&$=\sum a_{ij}x^i y^j$, equation in $K[x,y]$, not 0 or a monomial\cr
$f'_y$ & derivative of $f$ with respect to $y$; in \S\ref{sBaker}, \S\ref{sDiff} we assume \cr
&$f'_y\ne 0$, equivalently $K(C)/K(x)$ is separable\cr
$C$ & smooth projective curve over $K$ birational to $f=0$\cr
$\Delta$   & Newton polygon of $f$ over $K$ in $\R^2$, as in \S\ref{ssintromain} \cr
$\Delta_v$ & Newton polytope of $f$ over $O_K$ in $\R^3$, as in \S\ref{ssintromain}  \cr
$v$ & piecewise affine function $v: \Delta\to \R$ induced by $\vK$, as in \S\ref{ssintromain} \cr
$\partial$ & boundary (of a convex polygon in $\R^2$)\cr
$L$, $F$ & $v$-edges and $v$-faces of $\Delta$, the images of 1- and 2-dimensional\cr
   & faces of $\Delta_v$ under the homeomorphic projection $\Dv\to\Delta$ \cr
$\Delta(\Z)$ & interior$(\Delta)\cap\>\Z^2$; similarly $L(\Z)$, $F(\Z)$ \cr
$\bar\Delta(\Z)$ & closure$(\Delta)\cap\>\Z^2$; similarly $\bar L(\Z)$, $\bar F(\Z)$ \cr
$\delta_P, \delta_L, \delta_F$ & the denominator of $v(P)$ for $P\in\bar\Delta(\Z)$, and the common \cr
& denominators of $v(P)$ for $P\in \bar L(\Z)$ or $\bar F(\Z)$ (\ref{notintden})\cr
$f|_L$, $f|_F$ & restriction of $f$ to a $v$-edge or a $v$-face (\ref{defres}) \cr
$\overline{f_L}$, $\overline{f_F}$ & reduction to $k[t]$, $k[x,y]$ of the restrictions (\ref{defred})\cr
$X_L$, $X_F$ & affine schemes $\{\overline{f|_L}=0\}\subset\G_{m,k}$ and 
$\{\overline{f|_F}=0\}\subset\G_{m,k}^2$ (\ref{defXFXL})\cr
$\bar X_F$ & completion of $X_F$ with respect to its Newton polygon (\ref{defXFXL})\cr
\rnc{} & flat proper regular $O_K$-model with normal crossings\cr
  & of a curve over $K$ (exists in genus $\ge 1$, see \cite[9/1.7]{Liu})\cr
\mrnc{} & minimal \rnc{} model (unique, see \cite[9/3.36]{Liu})\cr
$\cC_\Delta/O_K$ & flat proper model of $C/K$ constructed in \S\ref{sProof},\cr
& \rnc{} if $C$ is $\Dv$-regular (cf. \ref{mainthm1}, \ref{mainthm2})\cr
$\cC_\Delta^{\min}/O_K$ & \mrnc{} model of a $\Dv$-regular curve (\ref{mainthm1}+\ref{minmain})\cr
$K^s$, $\bar K$ & separable and algebraic closure of $K$ \cr
$G_K$ & $\Gal(K^s/K)$, the absolute Galois group\cr
$I_K\normal G_K$ & the inertia group\cr
$\Iwild\normal I_K$ & the wild inertia subgroup\cr
$\Jac C$ & Jacobian of $C$\cr
$l$ & prime different from $p=\vchar k$ \cr
$\Sp_2$ & standard representation $I_K\to\GL_2(\Q_l)$, $\sigma\mapsto \smallmatrix 1{t(\sigma)}01$, where \cr
& $t: I_K\to\Z_l$ is the $l$-adic tame character (see \cite[4.1.4]{TaN})\cr
$\oplus, \ominus$ & direct sum and difference of representations\cr
$\triv$ & trivial representation\cr
$\G_{m,K}$ & $\Spec K[t,t^{-1}]$, affine line minus the origin \cr
$m_{*j}$, $m_{i*}$ & $j$th column and $i$th row of a matrix $M=(m_{ij})$ \cr 
$\tilde m_{*j}$, $\tilde m_{i*}$ & $j$th column and $i$th row of a matrix $M^{-1}=(\tilde m_{ij})$ \cr 
\end{tabular}

\medskip

\noindent
All representations are finite-dimensional, models are flat and proper, lattices are affine, 
and polygons are convex lattice polygons (vertices in $\Z^2$). 

See Table \ref{glossarytable} 
for examples of curves and 
the special fibres of their \mrnc{} models, with a glossary of notation used in the pictures. 



\begin{table}[!htbp]       
\par\noindent\hskip -1.7cm
\input{glossary.inc}
\par\bigskip
\caption{Examples and notation in the pictures. 
For the sample equations (right column) suppose $\vchar k=0$, say.}
\par\vskip 1cm
\label{glossarytable} 
\end{table}

\newpage

\marginpar{Use twiddles for reduction? Too many bars.}
\marginpar{\rIn}
\marginpar{Volume $\vol$ (or area? Or just ``not a line segment''?).}
\marginpar{Make $v$ into $\vK$ for the field}
\marginpar{Make, somewhat radically, $\cC$ into $C$ and $C$ into $C_K$? But not in Baker?}
\marginpar{0-, 1-, $v$-faces should be of $\Dv$, everywhere. Check.}

\endsection
\section{Baker's theorem}
\label{sBaker}

By a classical result going back to Baker in 1893 \cite{Bak}, under generic
conditions the genus of a plane curve is the number of interior integral 
points in its Newton polygon $\Delta$. Modern versions, over a general field, have 
been proved in \cite[Lemma 3.4]{KWZ}, \cite[Thm 4.2]{BP} and \cite[Thm 2.4]{Bee}.
There are extensions to higher-dimensional varieties and their invariants,
notably over~$\C$ \cite{Kush,Khov,BKK},\cite[\S4]{Bat}; 
see also \cite[Thm 1.3]{DL} over finite fields.
The question which curves admit a `Baker model' is addressed in detail in \cite{CV}.
Here we give a slightly different (and elementary) proof of Baker's theorem.
It describes regular differentials and points at infinity as well, and we will need
the steps later.

%
%
%
%
%

\begin{notation}
\label{notbaker}
Let $K$ be any field. Fix a polynomial, which is not zero or a monomial,
$$
  f = \sum_{i,j} a_{ij}x^iy^j \qquad \in K[x,y].
$$
Recall that the \emph{Newton polygon} $\Delta$ of $f$ is
$$
  \Delta = \text{convex hull}
  \bigl(\>  (i,j) \bigm|\>a_{ij}\ne 0\>\bigr)\subset\R^2.
$$
It is the smallest convex set containing the exponents of non-zero monomials,
and by assumptions it is a convex polygon with at least two vertices. 
Write

\medskip
\qquad
\begin{tabular}{lll}
$\Delta(\Z)$ &=& $\Z^2\>\cap$ interior of $\Delta$,\cr
$\bar\Delta(\Z)$ &=& $\Z^2\cap\Delta$.\cr
\end{tabular}
\smallskip\par\noindent

If $\Delta$ is not a line segment, we call its 1-dimensional faces \emph{edges}
%
We
use `$L\subset\partial\Delta$'
as a shorthand for `for $L$ an edge of $\Delta$'. For such an edge $L$, write
\begin{center}
$\LS\!=\!\LSsup{\Delta}\!=\!$ unique affine function $\Z^2\surjects\Z$ 
with $\LS|_L^{}\!=\!0$ and $\LS|_\Delta^{}\!\ge\! 0$.
\end{center}

\noindent
When $\Delta$ is a line segment, we view it as having 
two edges $L_1$, $L_2$ and $\Delta(\Z)=\emptyset$, and pick $L_i^*$ as above (not unique
in this case) so that $L_1^*=-L_2^*$.
\marginpar{Here edges are closed, but later $v$-stuff is open}

If $L\cap\Z^2$ = $\{(i_0,j_0),...,(i_r,j_r)\}$, ordered\footnote{The 
two possible orders give $f_L(t)$ vs $f_L(1/t)$; this choice does not affect 
anything} along $L$, set
\begin{center}
$
  f_L(t) = \sum_{n=0}^r a_{i_n,j_n} t^n \in K[t].
$
\end{center}
\end{notation}
%



%

\begin{theorem}
\label{thmbaker}
Let $f\!\in\!K[x,y]$, with Newton polygon $\Delta\!\subset\!\R^2$. If
$C_0: f=0$ is a smooth%
\footnote{smooth$=$non-singular over $\bar K$, and square-freeness is also smoothness; 
recall that $\Delta$=line segment is allowed (non-connected curves) but not $\Delta=$point or $\emptyset$
(not a curve)} curve in $\G_{m,K}^2$ and $f_L$ is square-free for every edge $L$, then
\begin{enumerate}
\item 
The non-singular complete curve $C$ birational to~$C_0$ has genus~$|\Delta(\Z)|$.
\item
\marginpar{See \cite{CDV}}
$C$ has a basis of regular differentials, non-vanishing on $C_0$,
$$
  \textstyle\omega_{ij}=\omegaij, \qquad   (i,j)\in\Delta(\Z).
$$
(If $f'_y\!=\!0$, equivalently $K(C)/K(x)$ is inseparable, then swap $x\!\leftrightarrow\!y$.)
\item
\marginpar{$K^s$ or $\bar K$ or $K^s$ or $K$ perfect? Does not matter in this theorem, probably, but matters
  for singular curves}
There is a natural bijection that preserves $\Gal(K^s/K)$-action,
$$
  C(\bar K)\setminus C_0(\bar K) \quad{\buildrel 1:1\over\longleftrightarrow}\quad 
    \coprod_{L\subset\partial\Delta} \text{\{roots of $f_L$ in $\bar K^\times$\}}.
$$
\item
If $P$ corresponds to a root of $f_L$ via (3), then
$$
\begin{array}{llllllll}
  {\rm (4a)} &&  \ord_P x^i y^j &=& \LS(i,j)-\LS(0,0) && \text{for}\>\> (i,j)\in\Z^2;\cr
  {\rm (4b)} &&  \ord_P \omega_{ij} &=& \LS(i,j)-1    && \text{for}\>\> (i,j)\in\Delta(\Z).
\end{array}
$$
\end{enumerate} 
\marginpar{Check whether this or outer-regularity is what is called $\Delta$-regular}
\end{theorem}

\begin{proof}
Let $C$ be the non-singular complete curve birational to $C_0$, possibly non-connected.
The statements are easy when $\Delta$ is a line segment 
($C_0=$union of $\G_m$s, $C$=union of $\P^1$s), 
so assume that vol$(\Delta)>0$.
We will cover $C$ by
charts $C_L$, $C_0\injects C_L\injects C$, indexed by the edges $L$ of $\Delta$.

First, let $\phi: (i,j)\mapsto mj-ni+s$ be any affine function $\Z^2\surjects \Z$ 
which is 
non-negative on $\Delta$ and zero on some vertex; so $l=(m,n)$ is a primitive 
generator of the line $L=\phi^{-1}(0)$, going counterclockwise around $\Delta$.%
\footnote{Thus $\phi=L^*$ in the case $L\cap\Delta$=edge}
Write $mn'-m'n=1$ with $m',n'\in\Z$.
In terms of $\XX=x^m y^n, \YY=x^{m'} y^{n'}$,
$$
  f(x,y) = \XX^? \YY^? \bigl(f_L(\XX) + \YY h(\XX,\YY)\bigr), \qquad h\in K[\XX^{\pm1},\YY].
$$
The univariate polynomial $f_L$ is a non-zero constant 
when $L\cap\Delta$=vertex, and is separable (by assumption) of positive degree 
when $L\cap\Delta$=edge. In both cases, 
$$
  C_L: f_L(\XX) + \YY h(\XX,\YY) = 0
$$
is a smooth curve in $\G_m\times\A^1$ with $C_L\cap \G_m^2= C_0$;
%
so $
  C_0\injects C_L\injects C
$.
The $\bar K$-points in $C_L\setminus C_0$ have $\YY=0$ and correspond to the roots of $f_L$
in $(K^s)^\times$, with the same Galois action. At these points,

\smallskip

\begin{tabular}{llllll}
\qquad&$(\XX=)$ & $x^m y^n$     & is regular non-vanishing, \cr
\qquad&$(\YY=)$ & $x^{m'} y^{n'}$ & is regular vanishing. \cr
\end{tabular}

Conversely, pick $P\in C\setminus C_0$. Either $x$ or $y$ has a zero or
a pole at $P$,~so
$$
  (i,j) \mapsto \ord_P x^i y^j 
$$
is a non-trivial linear map $\Z^2\to \Z$. It has a rank 1 kernel spanned by some primitive 
vector $(b,-a)$. Changing its sign if necessary, it follows from the above construction that $P$ is 
`visible' on $C_L$ obtained from $\phi: (i,j)\mapsto ai+bj+c$, for a unique $c$. Finally,
$$
  \LS(i,j)=ai+bj+c = \ord_P x^i y^j + c, \qquad \LS(0,0)=c.
$$ 
This proves (3) and (4a). 
%
%

(4b) Suppose $P$ corresponds to a root of $f_L$ via (3). 


It suffices to prove $\ord_P \omega_{ij}=\LS(i,j)-1$ 
for \emph{one} $(i,j)\in\Z^2$
and use (4a) 
to get all $(i,j)$. By Lemma \ref{diffGm}, this claim is 
invariant under $\SL_2(\Z)\subset\Aut\G_m^2$, reducing to the case
$$
  L\subset\text{$y$-axis}, \quad C_0: f_L(y)+xG(x,y)=0,
    \quad \omega_{11}=\tfrac{dx}{f_L'+xG'_y}.
$$
Since $f_L(y)$ is assumed to be separable and $P\in\text{$y$-axis}$,
the denominator of $\omega_{11}$ does not vanish at $P$.
Because $dx$ is regular non-vanishing on $\A^1\times\G_m$, we get
$\ord_P\omega_{11}=0$, as required.

%
%

(1), (2)
The total degree of any $\omega_{ij}$ is the sum of
$
  \ord_P \omega_{ij} = \LS(i,j)-1
$
over all $L\subset\partial\Delta$, taken with multiplicity $\deg f_L$. 
An elementary calculation (Lemma \ref{lemphisum}) shows that 
this sum is $2|\Delta(\Z)|-2$. This proves (1),
as the degree of a differential form on a complete genus $g$ curve is $2g-2$.
Finally, the $\omega_{ij}$ with $(i,j)\in\Delta(\Z)$ form a basis of regular
differentials, as there are~$g$ of them and 
there cannot be a linear relation between monomials $x^i y^j$ of degree $<\deg f$. 
(See Remark \ref{difind} for a slightly stronger statement.)
\end{proof}

\begin{remark}
\label{difind}
If $f=\sum a_{ij}x^i y^j$ satisfies the conditions of the theorem, 
the differentials
$\omega_{ij}$ with $(i,j) \in \bar\Delta(\Z)$
have exactly one linear relation between them up to scalars%
\footnote{It is not hard to see that this is a basis of differentials that are regular on $C_0$ 
  and have at most simple poles on $C\setminus C_0$, though we will not need this.}%
, namely $\sum a_{ij} \omega_{ij}=0$.
This is clear from the minimality (=square-freeness) of the defining equation $f=0$.
\marginpar{In particular, removing either a vertex or an edge from $\Delta\cap \Z^2$ 
results in linearly independent differentials.}
\end{remark}

\begin{example}
\label{ArkRab}
The following are genus 0 curves, with $\Delta(\Z)=\emptyset$:

\medskip
\qquad
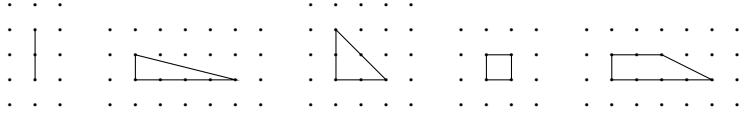
\begin{tikzpicture}
\dotgrid3 0024
\draw[-] (G11)--(G13);
\dotgrid3 40{10}3
\draw[-] (G51)--(G91)--(G52)--(G51);
\dotgrid3 {12}{0}{16}{4}
\draw[-] (G131)--(G151)--(G133)--(G131);
\dotgrid3 {18}{0}{21}{3}
\draw[-] (G191)--(G201)--(G202)--(G192)--(G191);
\dotgrid3 {23}{0}{29}{3}
\draw[-] (G241)--(G281)--(G262)--(G242)--(G241);
\end{tikzpicture}

\begin{center}
Example: Newton polygons of genus 0 curves (say char $K\nmid 2mn$)\\
$y^n=1$, \ $y=x^n\!+\!1$, \ $x^2\!+\!y^2=1$, \ $xy\!+\!x\!+\!y=1$, \ $x^n\!-\!2=y(x^m\!-\!1)$
\end{center}

\noindent
Conversely, Rabinowitz \cite[Thm~1]{Rab}, extending Arkinstall \cite[Lemma~1]{Ark} shows
that a convex lattice polygon $\Delta\subset\R^2$ with $\Delta(\Z)=\emptyset$ is 
$\GL_2(\Z)$-equivalent to one of the pictures above: a line segment, a right triangle 
with legs $1,n$ or $2,2$, or a trapezium of height 1 (right-angled as above, if desired).
\end{example}

\begin{example} 
Curves of genus $\le 4$ always have an equation satisfying Baker's theorem 
if $K$ is large enough, but not in general in higher genus \cite{CV}.
For instance, in genus 3, every smooth projective curve $C/K$ is 
either hyperelliptic or a plane quartic (but not both). 
If, say, $K=\bar K$, then $C$ has an equation $f=0$ with $\Delta$ either

\medskip
\begin{center}
\begin{tikzpicture}
\dotgrid4 0066
\draw[-] (G11)--(G51)--(G15)--(G11);
\node at (G22) {$\scriptstyle\bullet$};
\node at (G23) {$\scriptstyle\bullet$};
\node at (G32) {$\scriptstyle\bullet$};
\dotgrid{2.66} {9}0{19}4
\node at (2.5,0.8) {or};
\draw[-] (G101)--(G181)--(G103)--(G101);
\node at (G112) {$\scriptstyle\bullet$};
\node at (G122) {$\scriptstyle\bullet$};
\node at (G132) {$\scriptstyle\bullet$};
\end{tikzpicture}\\
\end{center}


\noindent
A basis of regular differentials is $\dxfy, x\dxfy, y\dxfy$ and 
$\dxfy, x\dxfy, x^2\dxfy$, respectively.
Similarly, in genus$\,\le\!2$ all curves have hyperelliptic models when $K=\bar K$,  
and there are seven (disjoint) types in genus 4 \cite[\S7]{CV}.
\marginpar{Regularity in $\G_m^2$ can always be satisfied, 
possibly over any field, with nodes 
at the boundary (or interior if one prefers).
Or locally, if nodes are rational (and $K$ is perfect?)}
\end{example}

\marginpar{
If $x^m y^n$ is a monomial, and $\vchar K\nmid mn$,
note on restriction to a sublattice? Need later that $\Dv$-regularity is stable
in tame covers.
}

%
%
%
%

\begin{remark}
\label{rmkbaker}
Let $C_0: f=0$ be an equation in $\G_{m,K}^2$, with $\vol(\Delta)>0$.

(a) If $C_0$ satisfies the conditions of Theorem \ref{thmbaker},
the proof shows that the associated smooth complete curve $C$ is covered by charts
\begin{equation}\label{Ccover}
  C = \bigcup_{L\subset\partial\Delta} C_L, \qquad C_L\,\injects\,\G_m\!\times\!\A^1\,\text{ closed}.
\end{equation}
These $\G_m\!\times\!\A^1$ cover, up to a finite set, 
a toric variety~$\TD$ with polytope $\Delta$, and $C$ 
is simply~the closure of $C_0$ in $\TD$.
It is often called the toric resolution of $C_0$ on $\TD$ (see \cite[\S2.1]{CDV},
\cite{Koe}, \cite{Cas})\marginpar{And other refs?} or tropical compactification \cite{Tev}. 
Two standard examples are

\begin{tabular}{l@{\qquad}l@{\quad$\supset$\quad}l}
\hbox to 5em{\hfill} $\Delta$ & $\TD$ & \hbox to 3em{\hfill} $C$ \cr
\hline
triangle on $(0,0)$, $(n,0)$, $(0,n)$ & $\P^2$ & curve of degree $n$,\Tcr
rectangle $[n,0]\times [0,m]$ & $\P^1\times\P^1$ & curve of degree $(m,n).$\cr
\end{tabular}

\smallskip
\noindent
(b) The theorem constructs $C_L$ as in \eqref{Ccover} and
the closure $C\subset \TD$ of $C_0$ for \emph{any} $f$ with Newton polygon $\Delta$, whether
$C_0$ is smooth or not. We call this scheme the \emph{completion $C_0^\Delta$ 
of $C_0$ with respect to its Newton polygon}.
Since $\TD$ is a complete variety, $C_0^\Delta$ is a proper scheme over $K$.

\smallskip
\noindent
(c) When $C_0$ is smooth, 
the squarefreeness condition in Baker's theorem to ensure smoothness of $C$ is sufficient
but not quite necessary. We say that $C_0$ is \emph{outer regular}
at an edge $L$ of $\Delta$ if $C_L$ is smooth. 

When $C_0$ is outer regular at all $L$, in other words when $C_0^\Delta$ is smooth,
the curve $C_0$ is often called nondegenerate with respect to $\Delta$ \cite[\S2, Def. 1]{CDV}. 
For example, $(x-1)^2+(y-1)^2=1$ is smooth in $\G_{m,\C}^2$,
and its closure in $\TD=\P^2_\C$ is smooth as well. The restriction of $f$
to the two edges of $\Delta$ on the coordinate axes is not squarefree, but $C_0$ is 
outer regular there.

\smallskip
\noindent
(d) Fix $\Delta$ with $\vol(\Delta)>0$ and let $V=\{\text{vertices of $\Delta$}\}$.
All $f\in\bar K[x,y]$ with Newton polygon $\Delta$
form a family $\cC\subset \cP\times\TD$ with a parameter variety
$\cP\iso \A^{\bar\Delta(\Z)\setminus V}\times\G_m^{V}$, flat over $\cP$.
Its generic member $f=0$ is smooth\footnote{This is subtle when $\vchar K>0$ 
  and fails in higher dimensions \cite[\S2, below Cor.\,1]{CDV}} 
\cite[\S2,\,Prop~1]{CDV}, \hbox{irreducible} (dimension count), and hence connected.
Consequently, \emph{any} $C_0$ with Newton polygon $\Delta$ is connected
(Zariski connectedness), and 
\begin{center}
  $C_0^\Delta$ has arithmetic genus $|\Delta(\Z)|$,
\end{center}
since arithmetic genus is constant in flat families  \cite[III.9.13]{Har}.
In particular, when $C_0$ is an integral scheme and so defines a curve, we have
\begin{center}
  geometric genus of $C_0$ $\>\>\le\>\>$ $|\Delta(\Z)|$ \qquad (Baker's inequality).
\end{center}
\end{remark}


\marginpar{
Singularities on a curve of arithmetic genus 0: aren't they automatically resolvable? 
Perhaps using models of P1's (always trees, same multiplcity etc.?)
}



\endsection
\section{Setup and main theorem}
\label{sRM}

From now on, $K$ is a field with a discrete valuation $v: K^\times \surjects \Z$,
ring of integers $O_K$,
uniformiser $\pi$ and residue field $k$. 
As in \S\ref{ssintromain}, fix 
$$
  f = \sum_{i,j} a_{ij}x^iy^j  \qquad \quad \in K[x,y],
$$
and consider its Newton polytopes over $K$ and over $O_K$,
$$
\begin{array}{lclrlll}
  \Delta & = & \text{convex hull}\bigl(\>  (i,j) &\bigm|\>a_{ij}\ne 0\>\bigr)  &\subset\R^2,\\[2pt]
  \Dv    & = & \text{lower convex hull}\bigl(\>  (i,j,\vK(a_{ij}))\!\!\! &\bigm|\>a_{ij}\ne 0\>\bigr)  &\subset\R^2\times\R.
\end{array}
$$
Assume from now on that $\vol(\Delta)\!>\!0$.
Above every point $P\in\Delta$ there is a unique point $(P,v(P))\in\Delta_v$. 
This defines a piecewise affine function $v: \Delta\to\R$, and
the pair $(\Delta,v)$ determines $\Dv$. 

%

\let\tD\Dv

\begin{definition}[Faces]
\label{deffaces} 
The images of the 0-, 1- and 2-dimensional (open) faces of the polytope $\tD$ under the 
homeomorphic projection $\tD\to\Delta$ are called 
\emph{$v$-vertices},
\emph{$v$-edges} and
\emph{$v$-faces} of $\Delta$. 
Thus, $v$-vertices are points in $\Z^2$, 
$v$-edges (denoted~$L$) are homeomorphic to an open interval,
and $v$-faces (denoted $F$) to an open disk. 
\marginpar{We will always have fixed $f$, $\Dv$ in mind, and say `$i$-faces' 
meaning $i$-faces of $\Delta$ with respect to $v$}
\end{definition}

\begin{notation}[Integer points, denominators] 
\label{notintden}
For $v$-edges and $v$-faces write 
$$
  L(\Z)=L\cap\Z^2, \qquad F(\Z)=F\cap\Z^2.
$$ 
As before, 
$\Delta(\Z)$ = (interior of $\Delta)\cap\Z^2$, and we write $\bar L(\Z)$, $\bar F(\Z)$, $\bar \Delta(\Z)$
to include points on the boundary. We decompose $\Delta(\Z)$ into
$$
\begin{array}{llllllllllll}
\Delta(\Z)^F &=& \bigl\{P\in\Delta(\Z) \bigm| P \in    F(\Z)\text{ for some $v$-face $F$}\bigr\},\\[2pt]
\Delta(\Z)^L &=& \bigl\{P\in\Delta(\Z) \bigm| P \notin F(\Z)\text{ for any  $v$-face $F$}\bigr\}.\\[1pt]
\end{array}
$$
We also use subscripts to restrict to points with $v(P)$ in a given set, such as
%
$$
\begin{array}{llllllllllll}
F(\Z)_{\Z_p} &=& \bigl\{P\in F(\Z) \bigm| p\nmid\den(v(P)) \bigr\}
             &=& \bigl\{P\in F(\Z) \bigm| v(P)\in\Z_p \bigr\}.
\end{array}
$$
(When $p=0$, the above is interpreted as an empty condition.)
We write $\Delta(\Z)^F_{\Z}$ etc. for multiple constraints. 
The \emph{denominator} $\delta_\lambda$ of a $v$-edge or a $v$-face $\lambda$ is, equivalently 
(see Lemma \ref{lemdelta} for (1)$\Leftrightarrow$(3)),
\begin{enumerate}
\item 
common denominator of $v(P)$ for $P\in\bar \lambda(\Z)$;
\item
smallest $m\ge 1$ for which $\bar \lambda(\Z)_{\oneoverZ{m}}=\bar \lambda(\Z)$;
\item
index of the affine lattice $\Lambda_\lambda$ spanned by $\bar \lambda(\Z)$ in its saturation in~$\Z^2$.
\end{enumerate}
If $F$ is a $v$-face, and we write 
\begin{center}
$v_F$ = unique extension of $v|_F$ to an affine function $\Z^2\to\Q$,
\end{center}
then $\Lambda_F=v_F^{-1}(\Z)$ and $\Z^2/\Lambda_F$ is cyclic of order $\delta_F$, 
again by Lemma \ref{lemdelta}. 
\end{notation}

\begin{example}
\label{exc1}
Let $C: f=0$ be the genus 4 curve defined by
$$
  f\>=\>y^3+\pi^2x^2y^2+y^2+\pi^3x^6+\pi x^3+\pi^3.
$$  
Here $\Dv$ has $v$-faces $F_1,F_2,F_3$ 
with $\delta_F$=3,3,6,
$v$-edges with $\delta_L$=1,1,1,2,3,3,3
and five $v$-vertices (from the terms of $f$ except for $\pi^2x^2y^2$ that has too high valuation).
Below (left) is a picture of $\Delta$ broken into $v$-faces, and 
the values of $v$ on $\bar\Delta(\Z)$. 
%
`Capital' ones indicate that the corresponding coefficient of $f$ has exactly 
that valuation (and not larger):
%
\smallskip
\noindent
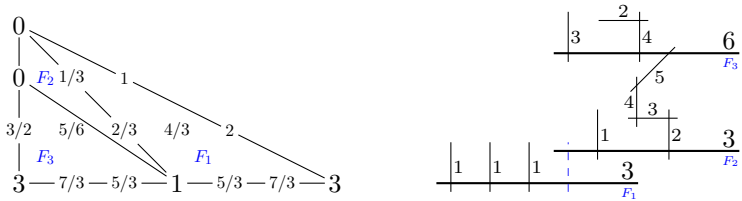
\marginpar{3D picture on the right?}
\marginpar{Into introduction, explaining the results? \'Etale cohomology,
 differentials, $\vchar k\ne 3$ and what happens when $\vchar k=3$ with
 shifts, Tate's algorithm and glueing}

\EXTHREEMODEL

\noindent
In this example, $|\Delta(\Z)|=4$, with 
$|\Delta(\Z)^F|=|\Delta(\Z)^L|=2$.
%
\end{example}

To produce the regular model
and describe the equations for the components of the special fibre (picture on the right),
we need to extract monomials from $f$ that come from a given face, restrict them to a sublattice,
and reduce the resulting equation:

%
%
%

\begin{definition}[Restriction]
\label{defres}
Let 
$g=\sum_{\ii\in\Z^n} c_\ii \xx^\ii\!\in\!K[\xx]$
be a polynomial in $n$ variables
$\xx\!=\!(x_1,...,x_n)$, and $\emptyset\ne S\subset\Z^n$. Take
$\Lambda$ to be the smallest affine lattice with $S\subset\Lambda\subset\Z^n$,
say of rank $r$. 
Choose an isomorphism $\phi: \Z^r\to\Lambda$, write $\yy\!=\!(y_1,...,y_r)$ 
and define the \emph{restriction}
$$
  g|_S = \sum_{\ii\in\phi^{-1}(S)} c_{\phi(\ii)} \yy^\ii.
$$
Different choices of $\phi$ change $g|_S$ by a $\GL_r(\Z)$-transformations 
of the variables and multiplication by monomials. 
For a $v$-vertex/edge/face $\lambda$,~set
$$
  f|_\lambda = f|_{\bar \lambda(\Z)_{\Z}}.
$$
\end{definition}

\begin{definition}[Reduction]
\label{defred}
Suppose $h\in K[x,y]$, and
there are $c,m,n\in\Z$ such that $\tilde h(x,y)=\pi^c h(\pi^m x,\pi^n y)$ 
has coefficients in $O_K$, and $h\in K[x,y]$ and $\tilde h\mod\pi\in k[x,y]$
have the same Newton polygon. Then we say that {$\tilde h\mod\pi$} 
is the \emph{reduction} of $h$, and denote it $\bar h$.
\end{definition}

\begin{example}
Let $f=y^2-x^3-\pi$, and $F$ the unique $v$-face of $\Delta$.
In the notation of \ref{defres} and \ref{defred}, the lattice
$\Lambda=\Z(3,0)+\Z(0,2)$ has index $\delta_F=6$ in $\Z^2$,
and $f|_F=f|_{\bar F(\Z)_{\Z}}=f|_\Lambda$ and its reduction $\overline{f|_F}$ are as follows:

\begin{center}
\EXRESTRICTION
\end{center}
\end{example}

\begin{definition}[Components $X_F$, $X_L$]
\label{defXFXL}
In the notation of \ref{defres}, \ref{defred}, define $k$-schemes
\begin{enumerate}
\item
$X_L\colon\{\overline{f|_L}=0\}\subset\G_{m,k}$ for each $v$-edge $L$ of $\Delta$,
\item
$X_F\colon\{\overline{f|_F}=0\}\subset\G_{m,k}^2$ for each $v$-face $F$ of $\Delta$,
\item
$\bar X_F$ = completion of $X_F$ with respect to its Newton polygon (\ref{rmkbaker}b); 
it is a geometrically connected scheme, proper over $k$ (\ref{rmkbaker}d).
\end{enumerate}
The vertices and edges of the Newton polygon of $\overline{f_F}$ 
correspond to those of~$F$ (cf. example above), and 
so there is a bijection from Baker's Theorem,
\begin{equation}
\label{xfpts}
  \bar X_F(\bar k)\setminus X_F(\bar k) \quad \longleftrightarrow \quad
  \coprod_{L\subset\partial F} X_L(\bar k).
\end{equation}
%
%
%
\end{definition}

We can now define $\Dv$-regularity. Modulo one allowed exception, 
it is simply saying that all $X_F$ need to satisfy Baker's Theorem assumptions:

\begin{definition}[$\Dv$-regularity]
\label{dvreg}
We say that $C$ (or $f$) is \emph{$\Dv$-regular} if all $X_F$, $X_L$ are smooth over $k$, 
\marginpar{Outer not defined yet}%
except that for $v$-edges $L\subset\partial F$ with $L\subset\partial\Delta$ and 
$\delta_L=\delta_F$, 
we drop the assumption that $X_L$ is smooth but require $\bar X_F$~to be 
outer regular at $L$, i.e. smooth 
at the points that correspond to $L$ via~\eqref{xfpts}.

It is not hard to see that regularity does not depend on the choice of $\pi$ and those 
made in \ref{defres}.
\end{definition}

\marginpar{By Baker, implies smoothness of the generic fibre? Was used in regularity proof.}

\begin{example}
\label{exc2}
The curve $C/K$ of Example \ref{exc1} 
is $\Dv$-regular iff $\vchar k\!\ne\!3$. 
(There is only one $v$-edge to check, all other restrictions are linear). 
%
%
%
\end{example}

\marginpar{Cannot write $F_1\cap F_2$ because $v$-faces are open. Check.}

\begin{remark}[Base change]
\label{dvbasechange}
Suppose $K'/K$ is an extension of discretely valued fields of ramification degree~$e$, 
with residue field extension $k'/k$.

For a $v$-face $F$,
write $X_F'$ for the scheme $X_F$ computed for $f$ as an element of $K'[x,y]$.
Let $\Lambda_F$ be the affine lattice spanned by $\bar F(\Z)_{\Z}$, 
so that $\Z^2/\Lambda_F$ is cyclic of order $\delta_F$ (see \ref{notintden}).
Let
$\Lambda_F\subseteq\Lambda_F'\subseteq\Z^2$ be  
the unique intermediate lattice with $(\Lambda_F':\Lambda_F)=\gcd(\delta_F,e)$. Let
$\phi: \G_{m,k'}^2=\Spec k'[\Lambda_F']\to \Spec k[\Lambda_F]=\G_{m,k}^2$ be the corresponding 
cover. Then 
$$ 
  X'_F \iso X_F \times_{\G_{m,k}^2,\phi} \G_{m,k'}^2,
$$
\marginpar{Heavy notation, make $X'$; and check outer regular condition}
and a similar statement holds for $v$-edges. Thus,
\begin{enumerate}
\item 
If $\gcd(\delta_F,e)=1$, then 
$X_F/k$ smooth $\iff X'_F$ is.
\item
If $p\nmid\gcd(\delta_F,e)$, then $\phi$ is \'etale, so
$X_F/k$ smooth $\Rightarrow X'_F$~is.
\item
If $p\nmid\gcd(\delta_F,e)$ for every $v$-face $F$ of $\Delta$, e.g. $K'/K$ is tame, then
\begin{center}
  $C/K$ is $\Dv$-regular \quad $\implies$ \quad $C/K'$ is $\Dv$-regular.
\end{center}
\end{enumerate}
\end{remark}


Finally, we need slopes of two $v$-faces $F_1$, $F_2$ meeting at a $v$-edge $L$.
They will control the chains of $\P^1$s between $\bar X_{F_1}$ and $\bar X_{F_2}$:

\begin{definition}[Slopes]
\label{defslopes}
\marginpar{Fix choices once and for all. This has the effect of making the 
  dual graph directed. Note on changing direction by reflecting $P$ and swapping 
  $F_1$ and $F_2$}
Every $v$-edge $L$ is either \emph{inner}, bounding two $v$-faces, say $F_1$ and $F_2$, 
or \emph{outer}, bounding $F_1$ only. Choose $P_0, P_1\in\Z^2$ with $\LSsup{F_1}(P_0)=0$,
$\LSsup{F_1}(P_1)=1$ (see \ref{notbaker}).
The \emph{slopes $[s_1^L,s_2^L]$ at $L$} are
$$
  s_1^L = \delta_L\>(v_1(P_1)\!-\!v_1(P_0)),  \quad\>\> s_2^L = 
            \bigleftchoice{\delta_L\>(v_2(P_1)\!-\!v_2(P_0))}{\text{for $L$ inner,}}
                      {\lfloor s_1^L-1\rfloor}{\text{for $L$ outer,}}
$$
where $v_i$ is the unique affine function $\Z^2\to\Q$ that agrees with $v$ on $F_i$.

\smallskip

In \S\ref{sProof} from $\Dv$ we construct a scheme $\cC_\Delta/O_K$. The reader is invited
to skip the construction completely since regularity and the special fibre of $\cC_\Delta$
can be described solely in terms of the above data. The results are:

\marginpar{How the choice of $P$ affects what}
\end{definition}

\begin{theorem}[Model in the $\Dv$-regular case]
\label{mainthm1}
Suppose $C: f=0$ is $\Dv$-regular.
Then $\cC_\Delta/O_K$ is a regular model of $C/K$
with strict normal crossings. Its special fibre $\cC_k/k$ is as follows:
\begin{enumerate}
\item 
Every $v$-face $F$ of $\Delta$ gives a complete smooth curve $\bar X_F/k$ of multiplicity $\delta_F$
and genus $|F(\Z)_{\Z}|$.
\item
For every $v$-edge $L$ with slopes $[s_1^L,s_2^L]$ 
pick $\frac{m_i}{d_i}\in\Q$ so that
$$
  s_1^L=\frac{m_0}{d_0}\!>\!\frac{m_1}{d_1}\!>\!\ldots\!>\!\frac{m_r}{d_r}\!>\!\frac{m_{r+1}}{d_{r+1}}=s_2^L
    \>\>\quad\text{with}\>\>\>
  \scalebox{0.9}{$\left|
  \begin{matrix}
  m_i\!\!\! & m_{i+1} \cr
  d_i\!\!\! & d_{i+1} \cr
  \end{matrix}
  \right|$}=1.
$$
Then $L$ gives $|X_L(\bar k)|$ chains of $\P^1$s of length $r$ from $\bar X_{F_1}$ to $\bar X_{F_2}$ 
(and open-ended in the outer case) and multiplicities $\delta_L d_1,...,\delta_L d_r$.%
\footnote{If $r=0$, this is interpreted as $\bar X_{F_1}$ meeting $\bar X_{F_2}$ transversally 
at $|X_L(\bar k)|$ points in the inner case, and no contribution from $L$ in the outer case.
In the outer case, this happens precisely when $\delta_L=\delta_{F_1}$.}
The group $G_k$ permutes the chains by its natural action on $X_L(\bar k)$.
\end{enumerate}
\end{theorem}

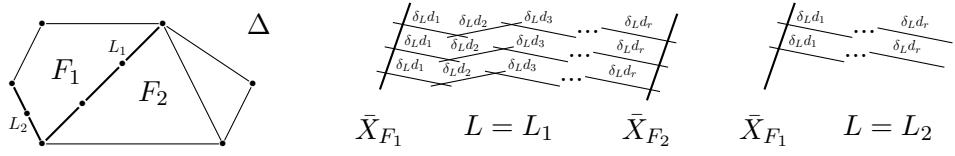
\CONTRONEFACEPIC

\begin{center}
Contributions to $\cC_k$ from an inner $v$-edge $L_1$ and an outer $v$-edge $L_2$
\end{center}

This is a special case of the following more general version:

\marginpar{Assume $\Delta$ has positive volume?}

\begin{theorem}[General case]
\label{mainthm2}
Suppose $f(x,y)=0$ defines a 1-dimensional scheme $C_0\subset\G_{m,K}^2$.
Then $\cC_\Delta/O_K$ is a proper flat model of the completion $C=C_0^\Delta$. 
Its%
\marginpar{Note: don't even assume generic fiber is smooth (can do arbitrary 
  nodal curves in Baker using shifts, though nodes need to be made rational,
  so in principle can handle more general curves in 3D as well)}
special fibre $\cC_k$ is a union of
closed subschemes $\bar X_F$ indexed by $v$-faces $F$ and $X_L\times_k \Gamma_L$
indexed by $v$-edges $L$, all 1-dimensional over $k$:
\marginpar{Check imperfect residue field}
\begin{enumerate}

\item
The $X_L, X_F, \bar X_F$ are as in \ref{defXFXL}. 
The $\bar X_F$ are geometrically connected, have arithmetic genus $|F(\Z)_{\Z}|$, and 
come with multiplicity $\delta_F$ in the special fibre
(i.e. defined by \smash{\scalebox{0.9}{$\overline{f_F}^{\delta_F}=0$}} in $\cC_k$).

\marginpar{Actually the theorem states that $\cC$ exists before 
we pick $n_i/d_i$}

\item
For every $v$-edge $L$ with slopes $[s_1^L,s_2^L]$ 
pick $\frac{n_i}{d_i}\in\Q$ so that
$$
  s_1^L=\frac{n_0}{d_0}\!>\!\frac{n_1}{d_1}\!>\!\ldots\!>\!\frac{n_r}{d_r}\!>\!\frac{n_{r+1}}{d_{r+1}}=s_2^L
    \>\>\quad\text{with}\>\>\>
  \scalebox{0.9}{$\left|
  \begin{matrix}
  n_i\!\!\! & n_{i+1} \cr
  d_i\!\!\! & d_{i+1} \cr
  \end{matrix}
  \right|$}=1.
$$
Let $\Gamma_L=\Gamma_L^1\cup...\cup\Gamma_L^r$ be a chain of $\P^1_k$s, with 
multiplicities $\delta_L d_i$, meeting transversally:
\marginpar{Can the statement and the proof be adapted to apply to 2IVstar-0 resolution? That is,
  with $\pi$ replaced by a more general $Z$}
$\infty\in\Gamma_L^i$ is identified with $0\in\Gamma_L^{i+1}$;\\
when $r=0$, let $\Gamma_L=\Spec k$, viewing it as $0\in \Gamma^1_L$ and $\infty\in \Gamma^r_L$.
\item
The subscheme $X_L\times\{0\}\subset X_L\times\Gamma^1_L$ is identified with the subscheme of 
$\bar X_{F_1}\setminus X_{F_1}$ that corresponds to $L$ via \eqref{xfpts}; similarly for
$X_L\times\{\infty\}\subset X_L\times\Gamma^r_L$ and $\bar X_{F_2}\setminus X_{F_2}$ when $L$ is inner.
\item
The intersections of $\bar X_{F_i}$ ($i=1,2$) with $X_L\times\Gamma_L$ and, when $r=0$, 
of $\bar X_{F_1}$ with $\bar X_{F_2}$ at $X_L$ are transversal.
In other words, the intersection, as a scheme, is given by \smash{\scalebox{0.9}{$\overline{f_L}^{\delta_L}=0$}}.
\end{enumerate}
The model $\cC_\Delta$ is given by explicit charts in \S\ref{sProof}.
It is geometrically regular at
\begin{itemize}
\item 
the smooth locus of $X_F$, for every $v$-face $F$,
\item
(smooth locus of $X_L$)$\times\Gamma_L$, for every $v$-edge $L$,
\item
the smooth points of $\bar X_{F}\setminus X_F$ that correspond to $L$ via \eqref{xfpts},
when $L\subset\partial F$ is outer with $\delta_L=\delta_F$, and $r=0$ in (2).
\end{itemize}
If $C_0$ is $\Dv$-regular, then $C$ is smooth, and
\marginpar{the generic fibre $\cC_K/K$ is semistable and?} 
$\cC_\Delta/O_K$ is its \rnc{} model.
\end{theorem}

The theorems are proved in \S\ref{sProof}. See Table \ref{glossarytable} for 
some examples, both $\Dv$-regular and not.
We end this section with a few comments:

\begin{remark}[Hirzebruch-Jung]
\label{sternbrocot}
To see that sequences as in \ref{mainthm2}(2) exist, take all numbers 
in $[s_2^L,s_1^L]\cap\Q$ of denominator $\le\max(\den s_1^L,\den s_2^L)$ in decreasing order. 
This is essentially a Haros (=Farey) series, 
and so satisfies the determinant condition \cite[Ch.\,III,\,Thm.\,28]{HW}. Now repeatedly remove,
in any order, terms of the form (\emph{loc.\,cit.},\,Thm.\,29)
$$
  \ldots > \frac{a}{b} > \frac{a+c}{b+d} > \frac{c}{d} > \ldots  
    \quad\longmapsto\quad
  \ldots > \frac{a}{b} > \frac{c}{d} > \ldots,
$$
until this is no longer possible; this 
corresponds to blowing down $\P^1$s of self-intersection $-1$, see \eqref{selfinteq} below.
The resulting minimal sequence is unique%
\footnote{E.g., \mrnc{} models would not be unique otherwise}.
If $(s_2^L,s_1^L)\cap\Z=\{N,...,N+a\}$ is non-empty, it has the form
$$
  s_1^L=\frac{n_0}{d_0}\!>\!\ldots\!>\!\frac{n_k}{d_k}\!>\!N\!+\!a\!>\!\ldots\!>\!
    N\!+\!1\!>\!N\!>\!\frac{n_l}{d_l}\!>\!\ldots\!>\!  
  \frac{n_{r+1}}{d_{r+1}}=s_2^L,
$$
with $d_0,...,d_k$ decreasing and $d_l,...,d_{r+1}$ increasing. If $s_2^L>0$, say (else shift 
by an integer), the numbers $N\!>\!\frac{n_l}{d_l}\!>\!\ldots\!>\!\frac{n_{r+1}}{d_{r+1}}$  
are the approximants of the Hirzebruch-Jung continued fraction expansion of $s_2^L$, 
see \cite[\S2.6]{Ful}.
(Similarly, for the first terms consider the expansion of $1-s_1^L$.)
\marginpar{See defslopes if changing signs is explained there. Explain when there are 
  no intermediate integers?}
These are designed to resolve toric singularities, which is essentially what the 
proof of \ref{mainthm2} does, so their appearance is expected. 
\marginpar{Finally, we remark that chains between two faces with $\delta_F=1$ have denominator 1.}
\end{remark}

\begin{remark}[Minimal vs normal crossings]
\label{minremark}
If $C/K$ is $\Dv$-regular, the theorem constructs a \rnc{} model, in fact \mrnc{} after
an easy modification (see \S\ref{sMinimal}). Recall that this is \emph{not} the same as the minimal regular 
model: the latter has generally fewer components but more complicated singularities in the 
special fibre. A \rnc{} model is encoded by the dual graph and the component multiplicities,
so it is somewhat easier to work with (and to draw, for the matter). Note that if a component $X$
of multiplicity $m$ meets $Y_1,...,Y_n$ of multiplicities $m_1,...,m_n$ transversally, then the 
self-intersection
\begin{equation}\label{selfinteq}
  X\cdot X = -\>\frac{m_1+...+m_n}{m} \qquad   (\in -\N).
\end{equation}
If $X$ has genus 0 and $X\cdot X=-1$, then $X$ may be blown down. However, when $n\ge 3$, the result 
is no longer r.n.c. For example, take the genus 12 curve $C/\Q_7$ from \cite{GRSS},
$$
  x^7=\frac{y^3\!-\!2y^2\!-\!y\!+\!1}{y^3\!-\!y^2\!-\!2y\!+\!1}.
$$ 
Letting 
$y\mapsto\frac{2y-3}{3y-2}$ and clearing the denominator, we get a model
(tree-like with no positive genus components, in agreement with \cite[App. A]{GRSS})
  \par\medskip\EXGRSS\hbox to 1.5em{\hfill}\par\smallskip\noindent  
\end{remark}

The components $X_{F_1}, X_{F_2}$ (multiplicity 21) have self-intersection $-1$. If they are blown
down, we find components of multiplicity 14 meeting `7,6 and 1', so they get self-intersection $-1$
and can be blown down again. Ditto for `6' and then `3', resulting in a minimal regular model 
which has one multiplicity 1 component with two (bad) singularities on the special fibre. 

%
%

%


\begin{example}
\label{expss}
Here is an example for Theorem \ref{mainthm2} (and not just \ref{mainthm1}).
Consider the curve $C_5/\Q_2$ from \cite[\S13.3]{PSS}. 
Letting $x\mapsto x\!+\!1, y\mapsto xy\!+\!1$ 
we find an equation
%

\begin{center}
$(3x^3\!+\!6x^2)y^3\!+\!9xy^2\!-\!(2x^3\!-\!6)y\!-\!4x^2\!-\!6x\!-\!4=0$.
\end{center}

\noindent
Here $\Dv$ and the special fibre of the model $\cC_\Delta$ are as follows:

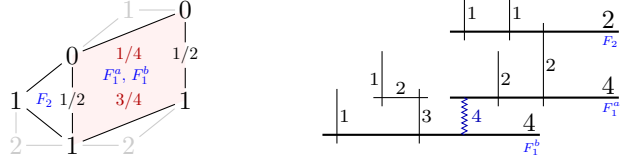
\begin{center}
\PSSCURVE
\end{center}

\marginpar{$X,Y\to$ some other names}

\noindent
At the $v$-face $F_1$ (in red), the reduced equation is%
\footnote{as clear from the picture of $\Dv$ on the left, thanks to the scarcity 
of units in $\F_2$}
$\overline{f_{F_1}}=\XX\YY+\XX+\YY+1$ and it gives two irreducible components $X_{F_1^a}: \XX=1$ and 
$X_{F_1^b}: \YY=1$ meeting transversally at $P=(1,1)$; 
both have multiplicity $\delta_{F_1}=4$ on the special fibre.
The point $P$ is not regular, and the regular model has a chain of $\P^1$s 
with multiplicity 4 from $X_{F_1^a}$ to $X_{F_1^b}$ (zigzag on the right picture). Its length 
is not determined by $\Dv$, and requires a computation in the local ring at $P$.
(It is actually 5, see \cite[\S13.3]{PSS}.) 
\end{example}

\marginpar{The same curve
$
  -2x^3y-2x^3+6x^2y+3xy^3-9xy^2+3xy-x+3y^3-y=0.
$
works immediately at $p=3$}

\begin{remark}[Reduction of points]
\label{redpts}
Let $\cC/O_K$ be any regular model of a curve $C/K$. Write $N$ for the set of 
multiplicity 1 components of the special fibre $\cC_k$ of $\cC$, and $X^{ns}\subseteq X$ for the smooth locus
of $X$ for $X\in N$.
There is a natural reduction map 
$$
  \red\colon C(K) \to \cC_k(k),
$$
landing in $\coprod_{X\in N} X^{ns}(k)$ (see \cite[9/1.32]{Liu}),
and onto it if $K$ is Henselian. 
(In particular, if $N=\emptyset$, then $C(K)=\emptyset$.)

Now suppose $C/K$ is $\Dv$-regular and $\cC=\cC_\Delta$. 
The components $X\in N$ come from three sources, and are indexed by:
\begin{itemize}
\item[($1_i$)]
triples $(L,r,n)$ where $L$ is an \emph{inner $v$-edge} with $\delta_L=1$,
$r$ a root of $\overline{f|_L}$ in $k^\times$, and $n\in (s_2^L,s_1^L)\cap\Z$, plus
\item[($1_o$)]
pairs $(L,r)$ where $L\subset\partial F$ is an \emph{outer $v$-edge}, $\delta_L=1$, $\delta_F>1$
and $r$ a root of $\overline{f|_L}$ in $k^\times$, plus
\item[(2)]
\emph{$v$-faces} $F$ with $\delta_F=1$.
\end{itemize}
\noindent
Write $C_0$ for the curve $f=0$ in $\G_m^2$, and let $P\in C(K)$ be a point.

(a) Suppose $P\in C(K)\setminus C_0(K)$. By Baker's theorem \ref{thmbaker}(3) it corresponds 
to a root $r$ of $f_L$ for an edge $L$ of $\Delta$. The edge breaks into $v$-edges that 
correspond to factors of $f_L$ by its Newton polygon over the completion of~$K$. The factor
that has $r$ as a root gives the $v$-edge indexing the component to which $P$ reduces
(Case $(1_o)$).

(b) Suppose $P\in C_0(K)$, so $x(P),y(P)\in K^\times$.
The linear form 
$$
  \cL_P(i,j,z) = v(x(P))i+v(y(P))j+z
$$
on $\Delta_v\subset\R^3$ is minimised on some face.
It is easy to see that this face must project down either onto a $v$-edge of $\Delta$ 
\marginpar{Needs to quote the proof. Relation of $\cL_P$ to $F^*$?}
(Case $(1_i)$ or $(1_o)$) or a $v$-face (Case $(2)$). It corresponds to the component to 
which $P$ reduces.
\end{remark}

\begin{example}
The curve in Table \ref{glossarytable}(i) has $|N|=2$, with one component 
from $F_2$ (of genus 1) 
and one from the leftmost $v$-edge (of genus 0). Here 
\begin{center}
\begin{tabular}{llllll}
$v(x(P))<0$ & $\Rightarrow$ & $\red(P) \in\ $genus 1 component,\cr
$v(x(P))>4$ & $\Rightarrow$ & $\red(P) \in\ $genus 0 component.
\end{tabular}
\end{center}

\noindent
There are no points with $v(x(P))\in\{0,1,2,3,4\}$, as all the corresponding $\cL_P$ are 
minimised on a vertex of $\Dv$; the $v$-edge $L$ between $F_1$ and $F_2$ 
gives no components since $(s_2^L,s_1^L)\cap\Z=\emptyset$.
Replacing $\pi^3$ by $\pi^n$ with $n\ge 5$
in the equation of $C$ would introduce $\lfloor\frac{n-1}{4}\rfloor$ components 
of multiplicity 1 on the chain of $\P^1$s from $L$, to which points can reduce as well.
\end{example}

\endsection
\section{Construction of the model $\cC_\Delta$}
\label{sProof}

We now construct the model $\cC_\Delta$ and prove Theorem \ref{mainthm2}.

If $p$ is harmless, $k=\bar k$ and $C$ is $\Dv$-regular,
one can get $\cC_\Delta$ by passing to $K_e=K(\sqrt[e]{\pi})$ with
$e=\lcm_F\delta_F$, showing that $C/K_e$ is semistable (easy), and taking the quotient of a
semistable model $\cC/O_{K_e}$ by the action of $\Gal(K_e/K)\iso\Z/e\Z$; say $\zeta_e\in K$. 
This approach is well studied \cite{Vie,LorD,CES,Hal}, 
and the theorem follows from here for $\Dv$-regular curves, but only in the \emph{tame} case
$p\!\nmid\!e$. 
In the wild case $p|e$, however, $C/K_e$
has no reason to be semistable. To deal with that, and with non-$\Dv$-regular curves, 
we take another approach and brutally construct $\cC_\Delta$ with a toroidal 
embedding over~$O_K$. The numerology of the fan is, however, exactly the same as 
from the quotient construction. In fact, the charts are explicit, and 
toric geometry is only needed to get a free proof that $\cC_\Delta$ is separated.

%
%
%

\stepr 1{Planes}
\marginpar{Relying only on $v$-vertices of $\Delta$; the toric variety in Baker also only depends
on the vertices of $\Delta$}
Take a $v$-edge, say inner, $L\subset \partial F_1, \partial F_2$. Let $\tilde F_1$, $\tilde F_2$, 
$\tilde L\subset\R^3$ be the planes/line through the faces of $\tD\subset\R^3$ above them.
Pick an `origin' $\tilde P_0\in\tilde L\cap\Z^3$. 
We construct vectors $\nu, \omega_0, \ldots, \omega_{r+1}\in\Q^3$ 
with $\omega_i$ lying above one another (same $x,y$ coordinates
and $v$-coordinate decreasing with $i$),~with
$$
  \tilde L=\tilde P_0\!+\!\R\nu,\quad
  \tilde F_1=\tilde P_0\!+\!\R\nu\!+\!\R\omega_0,\quad
  \tilde F_2=\tilde P_0\!+\!\R\nu\!+\!\R\omega_{r+1},
$$
$$
  \det\smallmatrix{\nu_x}{\omega_{i,x}}{\nu_y}{\omega_{i,y}}=1, \qquad 
  \det(\nu,\omega_i,-\omega_{i+1})=\tfrac{1}{\delta_L d_i d_{i+1}}.
$$
\vskip -3cm
\FONETWOPIC
\bigskip

\noindent
To be precise, let $L$ be any $v$-edge and $\tilde L$ as above. Write $\delta=\delta_L$.
Fix $P_0, P_1\in \Z^2$ with 
$\LS(P_0)=0$, $\LSsup{F_1}(P_1)=1$ as in \ref{defslopes} and fractions $\frac{n_i}{d_i}$ ($i=0,...,r\!+\!1$)
as in \ref{mainthm2}. Thus,
$$
  v_{F_1}(P_1)-v_{F_1}(P_0)=\frac{n_0}{\delta d_0}\!>\!\frac{n_1}{\delta d_1}\!
     >\!\ldots\!>\!\frac{n_r}{\delta d_r}\!>\!\frac{n_{r+1}}{\delta d_{r+1}}.
$$

Write $\tilde L=P_0+\R\nu$ with $\nu=(v_x,v_y,v_z)$ such that $\delta\nu\in\Z^3$ is
primitive and $(v_x,v_y)$ goes counterclockwise along $\partial F_1$.
Write $P_1-P_0=(w_x,w_y)$ and $\omega_i=(w_x,w_y,\frac{n_i}{\delta d_i})$;
the choices force $v_x w_y-v_y w_x=1$.
We also make one modification:
$$
\text{If $L$ is outer, redefine $d_{r+1}\!=\!0$, $n_{r+1}\!=\!-1$, $\omega_{r+1}\!=\!(0,0,-\tfrac1\delta)$.}
  \eqno{(\dagger)}
$$
These modified charts will be responsible for the `most outer' points of $\cC_\Delta$.

In all cases, define planes in $\R^3$ (see above picture, when $L$ is inner),
$$
  \cP_{L,i} = \R\nu + \R\omega_i \quad\qquad i=0,\ldots,r\!+\!1.
$$

\noindent
The planes $P_0+\cP_{L,i}$ rotate around $\tilde L$; the first one 
contains $\tilde F_1$, and the last either contains $\tilde F_2$ ($L$ inner) or is 
vertical ($L$ outer).
Let 
$$
  \cP_{L,i}^{\perp+} \subset \cP_{L,i}^{\perp}
$$
be the ray ($\iso \R_+$) for which the half-space 
$P_0 + \cP_{L,i} + \cP_{L,i}^{\perp+}$ contains $\Dv$.
%

\stepr 2{Toroidal embedding}
\marginpar{State what is our space and its dual, $M$ and $N$, and $\perp$}
For toroidal embeddings over a DVR, we follow \cite[\S IV.3]{KKMS}.
For a $v$-edge $L$, 
define polyhedral cones in $\R^2\times\R_+$,

\marginpar{$\perp$ should be $\vee$? Actually ok, because we are in $\R^2\times\R_+$}
\begin{tabular}{@{\qquad}l@{\quad}l@{\qquad}l}
0-dimensional cone  & $\sigma_0=\{0\}$  \Tcr
1-dimensional cones & $\sigma_{L,i}=\cP_{L,i}^{\perp+}$ &  ($0\le i\le r+1$) \cr
2-dimensional cones & $\sigma_{L,i,i+1}=\cP_{L,i}^{\perp+}+\cP_{L,i+1}^{\perp+}$ & ($0\le i\le r$). \Bcropt{2}
\end{tabular}

\noindent
The set $\Sigma$ of all such cones from all $L$ is a fan.
The associated toric scheme 
$$
  T_\Sigma=\bigcup_{\sigma\in\Sigma}T_\sigma, \qquad\quad T_\sigma=\Spec O_K[\sigma^\vee\cap\Z^3]
$$
(denoted $X_\Sigma=\bigcup X_\sigma$ in \cite{KKMS}) 
is flat and separated over $O_K$ (though not proper, as there are no 3-dimen\-sional cones).
Its generic fibre naturally contains $\G_{m,K}^2$, and we define 
\begin{center}
$\cC_\Delta$ $\>\>=\>\>$ closure of $\{f=0\}\subset \G_{m,K}^2$ in $T_\Sigma$. 
\end{center}


\stepr 3{Charts}
Let us describe the charts $T_\sigma$ explicitly. 
Fix $L$ as in \refstepr 1, and tweak $\nu, \omega_i, -\omega_{i+1}$ 
into a basis of $\Z^3$: take $\delta\nu$ (which is primitive) for the first basis vector.
For the other two, shift $d_i\omega_i$ by an appropriate multiple $k_i\nu$ to get
rid of the denominator in the last coordinate. In other words, pick $k_i\in\Z$ with
$$
  k_i\>\>\equiv\>\>-n_i \cdot (\delta v_z)^{-1}     \mod \delta \qquad\qquad   (0\le i\le r+1).
$$
%
%
As $\delta\nu$ is primitive, $\delta v_z$ is invertible mod $\delta$, 
\marginpar{By defn of $\delta_L$}
so this is posssible; 
when $\delta=1$, set $k_i=0$, say. For $0\le i\le r$ define
\begin{center}
\scalebox{0.9}{$
  M_{L,i}=
  \begin{pmatrix}
  \delta v_x\! & d_i w_x \!+\! k_i v_x\! & - d_{i+1} w_x \!-\! k_{i+1} v_x\cr
  \delta v_y\! & d_i w_y \!+\! k_i v_y\! & - d_{i+1} w_y \!-\! k_{i+1} v_y\cr
  \delta v_z\! & \quad\frac{n_i}\delta \!+\! k_i v_z\! & \quad\,-\frac{n_{i+1}}{\delta} \!-\! k_{i+1} v_z\cr
  \end{pmatrix}
     =
  \begin{pmatrix}
  \delta v_x & d_i w_x & -d_{i+1} w_x \cr
  \delta v_y & d_i w_y & -d_{i+1} w_y \cr
  \delta v_z & \frac{n_i}\delta & -\frac{n_{i+1}}{\delta}\\[3pt]
  \end{pmatrix}
  \begin{pmatrix}
  1\! & \frac{k_i}\delta\!\! & -\frac{k_{i+1}}\delta \cr
  0\! & 1 & 0 \cr
  0\! & 0 & 1 \cr
  \end{pmatrix} 
    = S T.
$}
\end{center}
\marginpar{Except we now need to redefine $n_{r+1}=-1$, $d_{r+1}=0$ in the outer case}
%
Our choices guarantee that $M=M_{L_i}$ has integer entries. Moreover,
$$
\scalebox{1.0}{$\begin{array}{llllll}
  \det M&=&\det S = \delta v_x \frac{-d_i w_y n_{i+1} + d_{i+1} w_y n_i}\delta
    - \delta v_y \frac{-d_i w_x n_{i+1} + d_{i+1} w_x n_i}\delta + \delta v_z \cdot 0\cr
      &=& \delta v_x \frac{w_y}{\delta} - \delta v_y \frac{w_x}{\delta} = 1.
\end{array}$}
$$
Thus $M\in\SL_3(\Z)$, with inverse $M^{-1} = (\tilde m_{ij}) = T^{-1} S^{-1}$,
\begin{equation}\label{Minveqn}
\scalebox{0.9}{$
  M_{L,i}^{-1} = 
  \begin{pmatrix}
  1 & -k_i/\delta & k_{i+1}/\delta \cr
  0 & 1 & 0 \cr
  0 & 0 & 1 \cr
  \end{pmatrix} 
  \begin{pmatrix}
  w_y/\delta & -w_x/\delta & 0 \cr
  n_{i+1} v_y \!-\! \delta d_{i+1} v_z w_y\! & -n_{i+1} v_x\!+\!\delta d_{i+1} v_z w_x\! & \delta d_{i+1} \cr
  n_i v_y \!-\! \delta d_i v_z w_y & -n_i v_x \!+\! \delta d_i v_z w_x & \delta d_i \cr
  \end{pmatrix}$}.
\end{equation}
Because of the choices of $\nu, \omega_i$ and their orientation, we have
$$
\begin{array}{l@{\>=\>}l@{\>=\>}l@{\qquad}l@{\>=\>}l@{\>=\>}l}
  \cP_{L,i} & \R\nu + \R\omega_i & \R\> m_{*1}+ \R\> m_{*2}, & \sigma_{L,i} & \cP_{L,i}^{\perp+} & \R_+\>\mt3*, \cr
  \cP_{L,i+1} & \R\nu + \R\omega_{i+1} & \R\> m_{*1}+ \R\> m_{*3}, & \sigma_{L,i+1} & \cP_{L,i+1}^{\perp+} & \R_+\>\mt2*, \cr
\multicolumn{6}{c}{\cP_{L,i}\cap\cP_{L,i+1}=\R\nu= \R\> m_{*1}, \quad \sigma_{L,i,i+1} = \langle \cP_{L,i}^{\perp+},\cP_{L,i+1}^{\perp+}\rangle = \R_+\>\mt2* \!+\! \R_+\>\mt3*,}
\end{array}
$$
and monomial exponents from the dual cones are
$$
\begin{array}{llllllllllll}
  \sigma_{L,i}^\vee \cap \Z^3 &=& \Z m_{*1} + \Z m_{*2} + \Z_+ m_{*3} \cr
  \sigma_{L,i+1}^\vee \cap \Z^3 &=& \Z m_{*1} + \Z_+ m_{*2} + \Z m_{*3} \cr
  \sigma_{L,i,i+1}^\vee \cap \Z^3 &=& \Z m_{*1} + \Z_+ m_{*2} + \Z_+ m_{*3}. \cr
\end{array}
$$
%
%
In other words, the coordinate transformation 
$$
\begin{array}{l@{\>\>}l@{\>\>}l@{\>\>}l@{\>\>}l@{\>\>}l@{\>\>}l@{\>\>}l}
(\XX,\YY,\ZZ) &{=}& 
  ({x^{m_{11}}y^{m_{21}}\pi^{m_{31}}},
      {x^{m_{12}}y^{m_{22}}\pi^{m_{32}}},
      {x^{m_{13}}y^{m_{23}}\pi^{m_{33}}})
&{=}& (x,y,\pi) \bullet M,\cr
(x,y,\pi) &{=}& (\XX,\YY,\ZZ) \bullet M^{-1}
\end{array}
$$
gives ring homomorphisms (as $\mt23,\mt33 \ge 0$, the denominator makes sense)
$$
  K[x^{\pm1},y^{\pm1}]
  \quad 
  {\buildrel \scriptscriptstyle M\over \iso}
  \quad 
  \frac{O_K[\XX^{\pm1},\YY^{\pm1},\ZZ^{\pm1}]}{(\pi-\XX^{\tilde m_{13}}\YY^{\tilde m_{23}}\ZZ^{\tilde m_{33}})}
  \>\>\longleftinjects\>\>
  \frac{O_K[\XX^{\pm1},\YY,\ZZ]}{(\pi-\XX^{\tilde m_{13}}\YY^{\tilde m_{23}}\ZZ^{\tilde m_{33}})} = R,
$$
and this is the inclusion $O_K[\sigma_0^\vee]\leftinjects O_K[\sigma_{L,i,i+1}^\vee]$ of \cite{KKMS}.
Thus,
\marginpar{
Separate lemma on the reduced equation for a general linear form --- useful for points, 
differentials, regularity etc. (even properness)?}
$$
\begin{array}{llllllllllll}
  T_{\sigma_{L,i,i+1}}&=&\Spec R, &&&
  T_{\sigma_{L,i+1}}&=&\Spec R[Z^{-1}],\cr
  T_{\sigma_{L,i}}&=&\Spec R[Y^{-1}], &&&
  T_{\sigma_0}&=&\Spec R[Y^{-1},Z^{-1}],\cr
\end{array}
$$
and these are all flat $O_K$-schemes of relative dimension 2.
%
%
Glueing the $T_{\sigma_{L,i,i+1}}$ for varying $L$ and $i$ along their common 
opens gives $T_\Sigma$.

\stepr 4{The closure of $f=0$ in $T_\Sigma$}
The model $\cC_\Delta$ is covered by schematic closures of $\{f=0\}\subset \G_{m,K}^2$ 
in the various $T_{\sigma_{L,i,i+1}}$. 
Explicitly, view $f$ as a polynomial in $x,y,\pi$ 
whose non-zero coefficients $u_{ij}$ are in $\smash{O_K^\times}$,
\begin{equation}\label{fFMeq}
  f = \sum_{i,j} u_{ij} \pi^{v(a_{ij})}x^i y^j, \qquad u_{ij}=\frac{a_{ij}}{\pi^{v(a_{ij})}}.
\end{equation}
Let $M=(m_{ij})\in\SL_3(\Z)$, for the moment arbitrary. Write $M^{-1}=(\mt ij)$, 
suppose $\mt23\ge 0,\mt 33>0$, and change variables 
$x=\XX^{\tilde m_{11}}\YY^{\tilde m_{21}}\ZZ^{\tilde m_{31}}, y=\ldots, \pi=\ldots$ 
according to $M^{-1}$. 
For some unique $n_Y, n_Z\in\Z$, we have
\marginpar{Explicify stars}
$$
  f = 
  \YY^{n_Y} \ZZ^{n_Z} \cF_M(\XX,\YY,\ZZ), \quad \cF_M\!\in\! O_K[\XX^{\pm 1},\YY,\ZZ],
  \>\>Y\!\nmid\!\cF_M,\>\>Z\!\nmid\!\cF_M.
$$
Thus $\cF_M$ and $f$ have the same coefficients $u_{ij}$, just different monomials,
and $\cF_M=0$ defines a complete intersection
$$
  \cF_M(X,Y,Z) = \pi\!-\!\XX^{\mt13}\YY^{\mt23}\ZZ^{\mt33} = 0 \eqno{(*)}
$$
in $\Spec O_K[\XX^{\pm 1},\YY,\ZZ]$. 
Now let $M\!=\!M_{L,i}$ from \refstepr3. This $\Spec$ is 
$T_{\sigma_{L,i,i+1}}$, and $(*)$ defines
the open set $\cC_\Delta\cap T_{\sigma_{L,i,i+1}}\subseteq\cC_\Delta$. 
These cover all of~$\cC_\Delta$.

\stepr {45}{Special fibre}

The special fibre $\pi\!=\!0$ of $T_\Sigma$ meets $(*)$ in $\YY^{\mt23}\ZZ^{\mt33}\!=\!0$.
Its underlying reduced subscheme is $Z=0$ in the case $(\dagger)$, and 
$YZ=0$ otherwise. To find $(*)\cap \{Z=0\}$, 
write $f = Y^{n_Y} Z^{n_Z} \cF_M$ as above. 
Noting that $Z=0\,\implies\,\pi=0$, we obtain a $k$-scheme
$$
  (*)\cap \{Z=0\} = \Spec k[X^{\pm 1},Y] / (\cF_M(X,Y,0)).
$$
The non-zero terms in $\cF(X,Y,0)\in k[X,Y]$ come from monomials $\pi^{v(a_{ij})} x^i y^j$ in \eqref{fFMeq} that,
when rewritten in $X,Y,Z$, have minimal exponent (=$n_Z$) of $Z$. 
%
In other words, they minimise the following linear form $\Z^3\to \Z$ on $\Dv$,
$$
  \phi\colon (\alpha,\beta,\gamma) \>\>\mapsto \>\> \ord_Z(x^\alpha y^\beta \pi^\gamma) 
    = \mt31 \alpha + \mt32 \beta + \mt 33 \gamma.
$$
Geometrically, $\cP=\ker\phi$ is the plane spanned by the two columns 
of~$M$, and $\phi$ is minimised on some face $\tilde\lambda\subset\Dv$ parallel to $\cP$. 
Let $\lambda$ be the projection of $\tilde\lambda$ onto $\Delta$; it is 
a $v$-vertex/edge/face. With 
an appropriate choice of the basis of the lattice in \ref{defres}, \ref{defred}, 
we find that
$$
  \bar\cF_0(X,Y) = \overline{f_\lambda}. 
$$
This, and an identical computation for $(*)\cap \{Y=0\}$ to $M\!=\!M_{L,i}$ gives
$$
\begin{array}{llllllllllll}
\cF_M(X,0,0) &=& \overline{f_L}(X) \cr
\cF_M(X,Y,0) &=& \bigleftchoice{\overline{f_{F_1}}(X,Y)}{i=0}{\overline{f_L}(X)}{i>0}\cr
\cF_M(X,0,Z) &=& \bigleftchoice{\overline{f_{F_2}}(X,Y)}{i=r,\>\>L\text{ inner}}{\overline{f_L}(X)}{i<r.}
\end{array}
\eqno{(**)}
$$
For a fixed $v$-face $F$ and varying $L\subset\partial F$, the subsets
$\overline{f_F}=0$ of $\G_{m,k}\times\A^1_k$ are precisely the charts that 
define the completion $\bar X_F$. Glueing the pieces together, we find that $\cC_k$ 
is the union of $\bar X_F$ over all $v$-faces $F$, and $X_L\times \Gamma_L$ over $v$-edges $L$, as 
claimed in Theorem \ref{mainthm2} (and correct multiplicities). 
\marginpar{This needs $\delta_F=\delta_L v_F(P)$ as a separate remark near \ref{defslopes}}
Since $X_L$ are finite and so proper over $k$, and the $\bar X_F$
and $\Gamma_L$ are proper as well, $\cC_k/k$ is proper.
By \cite[3/3.28]{Liu} (or EGA IV.15.7.10), $\cC_\Delta/O_K$~is~proper.

\stepr 5{Regularity}

Let $P\in\cC_k(\bar k)$ be a point. It corresponds to an ideal 
$$
  m_P=(\pi,\XX\!-\!x_0,\YY\!-\!y_0,\ZZ\!-\!z_0)\>\subset\> O_K[\XX^{\pm 1},\YY,\ZZ] / (*).
$$
To prove the claims in \ref{mainthm2} when $P$ is regular, we compute $\dim m_P/m_P^2$. 
The equation $\pi=\XX^{\mt13}\YY^{\mt23}\ZZ^{\mt33}$ of $(*)$ eliminates $\pi$ 
(recall that $\mt23,\mt33\ge 0$),
and the assumptions of \ref{mainthm2} on $\bar f_L, \bar f_F$ eliminate either $X$ or $Y$ from 
$m_P/m_P^2$; so $\cC$ is regular at $P$ under those assumptions.

Finally, suppose $C$ is $\Dv$-regular.
Then $\cC$ is geometrically regular at every point of $\cC_k$, as above.
As geometric regularity is an open condition, this implies that the generic fibre,
which is $\bar C$ by construction, is smooth. So $\cC$ is a regular model of $\bar C$
in this case. 

This proves Theorems \ref{mainthm2} and \ref{mainthm1}.
%
%
%
\marginpar{Glueing: given functions $f_1,...,f_n$ on the generic fibre, conditions
  to get a regular model from regular models of affine pieces $f_i\ne 0,\infty$? this is 
  also nice for the glueing section later}
\qed



\endsection
\section{Minimal model}
\label{sMinimal}

From now on, assume that $C$ has genus $\ge 1$.

Here, in Theorem \ref{minmain}, we describe the \emph{minimal regular normal crossings} (\mrnc) model,
\marginpar{Models of $\P^1$: all faces are removable. Genus $>0$ from here onwards; 
  mention in notation?}
a canonical regular model obtained from any normal crossings model by 
`removing unnecessary stuff', that is repeatedly blowing down genus 0 components of 
self-intersection $-1$.
In our case, `unnecessary stuff' comes from $v$-faces that avoid $\Delta(\Z)$.
We will see later that $\Delta(\Z)$ also plays a crucial role 
in the description of \'etale cohomology and differentials.

\marginpar{mainthm1, mainthm2 $\to$ A, B?}

\begin{definition}
\label{defprincipal}
Let $F$ be a $v$-face of $\Delta$. We say that
\begin{itemize}
\item 
$F$ is \emph{removable} if $F\cap\Delta(\Z)=\emptyset$.
\item
$F$ is \emph{contractible} if $F\cap\Delta(\Z)$ has a unique point $P$, all 
$v$-edges $L\subset\partial F$ with $P\notin L(\Z)$ are outer, and $v(P)\in\Z$.
\item
$F$ is \emph{principal} in all other cases.
\end{itemize}
\end{definition}

We will refer to chains of $\P^1$s from Theorem \ref{mainthm1}(2) as \emph{inner chains} or
\emph{outer chains}, depending on whether they come from an inner or outer $v$-edge.

\marginpar{Removable faces can be locally removed with shifts. That makes our model 
more canonical. However these shifts affect the rest of the model,
therefore we need to incorporate glueing first. And in the case of elliptic 
curves the classification becomes cleaner.}

\begin{proposition}
\label{propremcon}
Let $C/K$ be a $\Dv$-regular curve, 
and $\bar X_F$ a component of $\cC_k$ from a $v$-face $F$, as in \ref{mainthm1}.
The following conditions are equivalent:
\begin{itemize}
\item[(a$_0$)]
$\bar X_F$ has genus 0 and meets the rest of the special fibre either at one point or 
at two points one of which lies on an outer chain.
\item[(b$_0$)]
$F$ is $\GL_2(\Z)$-equivalent to an $n\times 1$ right-angled triangle whose short sides 
are either outer $v$-edges (as in Table \ref{remcontable}, top left) or one outer $v$-edge 
and one inner, of length 1, bounding a removable face.
\item[(c$_0$)]
$F$ is removable.
\end{itemize}
Similarly, the following conditions are equivalent:
\begin{itemize}
\item[(a$_1$)]
$\bar X_F$ has genus 0 and meets the rest of the special fibre at two points,
that are not outer chains.
\item[(b$_1$)]
$F$ is $\GL_2(\Z)$-equivalent to a polygon in Table \ref{remcontable}, bottom row
(an $n\times 1$ or $2\times2$ right-angled triangle, 
\marginpar{Check outer means outer and not removable}
a chopped version of it, or a right-angled trapezoid with parallel sides $1$,$n$),
with inner $v$-edges as shaded and of denominator 1.
\item[(c$_1$)]
$F$ is contractible.
\end{itemize}
\end{proposition}

\begin{table}[!htbp]
\begin{tabular}{l@{\qquad}l@{\qquad}ll}
\cbox{Removable\\$v$-face $F$:}
& \cbox{\remdeltapic} & $\iff$ & \cbox{\remmodelpic}\\[13pt]
\cbox{Contractible\\$v$-face $F$:}
& \cbox{\condeltapic} & $\iff$ & \cbox{\conmodelpic}\\[16pt]
\end{tabular}
\caption{Removable and contractible faces. Shading indicates the direction towards 
  other parts of $\Delta$ (inner $v$-edges)}
\label{remcontable}
\end{table}

\marginpar{Note that contractible trapezoids might be singular, not satisfying
$\Dv$-regularity condition}

\begin{proof}
\footnote{A direct classification-free proof of (a$_i$)$\iff$(c$_i$) would be very much appreciated}
Denote the four shapes in (b$_1$) by $S_1, S_2, S_3, S_4$ (cf. Table \ref{remcontable}). 
The edges of $S_3$ and $S_4$ span $\Z^2$, so such faces have $\delta_F=1$; 
similarly $\delta_F|n$ for $S_1$ and $\delta_F|2$ for $S_2$.

(b$_i$)\implies(c$_i$) Clear.

(b$_i$)\implies(a$_i$) By Theorem \ref{mainthm1}, outer $v$-edges contribute nothing, and inner 
ones give one chain ($S_1$, $S_4$) or two ($S_2$, $S_3$).

(a$_i$)\implies(b$_i$) 
Depending on the case, we have
\begin{equation}\label{chain0count}
  1\>\text{ or }\>2 \>\>=\>\> \text{\#$\>$chains from $X_F$} 
    \>\>=\>\> \sum_L |X_L(\bar k)|
    \>\>=\>\> \sum_L |L(\Z)_{\Z}|,
\end{equation}
summing over inner $v$-edges $L$, and outer $v$-edges $L\subset\partial F$ 
with $\delta_L<\delta_F$ (see Theorem \ref{mainthm1}).

Write $\tilde F$ for $F$ restricted to the sublattice $\Lambda$ spanned by 
$\bar F(\Z)_{\Z}$. The genus of $X_F$ is $|F(\Z)_{\Z}|$, assumed to be 0.
Therefore, by the Arkinstall--Rabinowicz classification of convex polygons with 
no interior points (Example \ref{ArkRab}), $\tilde F$ is one of $S_1,...,S_4$.
As $\Z^2/\Lambda$ is cyclic of order $\delta_F$ (see \ref{notintden}), 
$$
  \begin{tabular}{lll}
  $\delta_F =1$ &\implies& all $\delta_L=1$\cr
  $\delta_F >1$ &\implies& $\delta_L<\delta_F$ for all $L$ except possibly those\cr
    && parallel to one fixed line in $\Z^2$.\cr
  \end{tabular}
  \eqno(*)
$$
Suppose LHS of \eqref{chain0count} is 1. 
Then $F$ has one inner face $L$ with $L(\Z)_{\Z}=\emptyset$, and 
all the others are outer with $\delta_L=\delta_F$. This eliminates $\tilde F=S_2$. In the $S_3$ 
and $S_4$-case, $\delta_F=1$ by ($*$), but then one inner face with $L(\Z)_{\Z}=\emptyset$ 
is impossible by convexity, as there is no space behind $L$ to contain interior integral points 
(while genus$(C)>0$). So $\tilde F$ is $S_1$, the inner face gives the unique chain of $\P^1$s, and two
outer ones have $\delta_L=\delta_F$ which forces $\delta_F=1$ by ($*$). Thus, $F$ is $S_1$ 
(and, moreover, $n=1$). 

If the LHS of \eqref{chain0count} is 2, then the same (but more tedious) considerations give
the list of possible cases.

(c$_i$)\implies(b$_i$) Similar: as $F$ has no interior integer points, it is one 
of the $S_i$ by the Arkinstall--Rabinowicz classification; the rest is by inspection. 
\end{proof}

\marginpar{In genus 0 no interior points, so every $v$-face is removable}

%

\begin{remark}
\label{remprin}
Let $C/K$ be $\Dv$-regular curve of genus $g>0$.
\begin{itemize}
\item 
If $g=1$ with $|\Delta(\Z)^L_{\Z}|=1$, there are no principal $v$-faces.
\item
In all other cases, e.g. if $g>1$, there is always a principal $v$-face.
\end{itemize}
Indeed, in the first case, it follows from Theorem \ref{mainthm1}, after contracting the chains, that the special
fibre is a reduced $n$-gon of $\P^1$s (Kodaira type I$_n$).
In the second case, if $\Delta(\Z)^F\ne\emptyset$ or if $g=1$, $\Delta(\Z)^L_{\Z}=\emptyset$,
then there is a principal face by definition.
Otherwise for distinct $P,Q\in\Delta(\Z)$,
every point in the interior of the line segment $PQ$ belongs to principal $v$-face(s).
\end{remark}

\marginpar{This is in effect $C$ semistable $\iff \Jac C$ semistable with Saito thrown in}

\begin{proposition}
\label{propdefic}
Let $C$ be a $\Dv$-regular curve of genus $g>0$. If $v(P)\in\Z$ for every $P\in\Delta(\Z)$,
then exactly one of the following holds:
\begin{itemize}
\item 
$g>1$ and every principal $v$-face $F$ has $\delta_F=1$.
\item
$g=1$ and $C$ is a semistable elliptic curve (type ${\rm I}_n$ with $n\ge 0$).
\item
$g=1$, $C(K)=\emptyset$, and up to an isomorphism of the form $x\mapsto x^a y^b\pi^c$, 
$y\mapsto x^h y^i\pi^j$, $\Dv$ has one of the three exceptional shapes

\begin{equation}\label{eqdefshapes}
\hbox to 25em{\rm\DEFICIENTSHAPES}
\end{equation}

\noindent
The special fibre $\cC_k$ consists of a unique 
genus~1 component $\bar X_F$, of multiplicity 2,2,3, respectively.
%
%
\end{itemize}
\end{proposition}

\begin{proof}
Case 2. $g=1$, $|\Delta(\Z)^L_{\Z}|=1$. See Remark \ref{remprin}.

Case 3. $g=1$, $|\Delta(\Z)^F_{\Z}|=1$. There is a unique principal $v$-face $F$.
Up to $\GL_2(\Z)$-equivalence, there are 15 possible convex lattice polygons 
with one interior point \cite[Thm. 5]{Rab}. Of these, in 12 cases the coordinates of the 
vertices and the interior point generate $\Z^2$ as an affine lattice, so $\delta_F=1$. 
The other three shapes (Cases 1,4,9 of \emph{loc.cit.}) are as in \eqref{eqdefshapes},
with lattice index 2,2,3. Assume we are in one of these, so $\delta_F=2,2,3$, respectively. 
After rescaling $x$, $y$ by powers of $\pi$, the valuations can be made as in \eqref{eqdefshapes}.

Case 1. $g\ge 2$. Suppose $F$ is a principal $v$-face with $\delta_F>1$. 
If $|F(\Z)|<2$, from \ref{ArkRab} and the proof of Case 3, the only possibilities for $F$ 
are 
\begin{itemize}
\item[(a)]
an $n\times 1$ right triangle; here $|F(\Z)|=0$, $1\!<\!\delta_F\,|\,n$,
\item[(b)]
a $2\times 2$ right triangle; here $|F(\Z)|=0$, $\delta_F\!=\!2$,
\item[(c)]
the three cases in \eqref{eqdefshapes}; here $|F(\Z)|=1$, $\delta_F\!=\!2,2,3$ respectively.
\end{itemize}
In (b) and (c), as $F$ is not the whole of $\Delta$, at least one $v$-edge $L\subset\partial F$ 
must be inner with $\delta_L>1$, contradicting $\Delta(\Z)=\Delta(\Z)_{\Z}$.
Similarly, in (a), as $F$ is not removable or contractible, one of the corners 
of the $v$-edge $L\subset\partial F$ of length $n$ must in $\Delta(\Z)$, and this again
makes $L$ inner with $\delta_L>1$.

Finally, if $|F(\Z)|\ge 2$, subdivide $F$ into convex lattice polygons one of which
has exactly one interior integer point, and apply the same argument to it to show that it is in 
(c), to get the same contradiction.
\end{proof}

\marginpar{
Pictures: every thick (=principal) component has $\ge 3$ chains or positive genus, every thin
  one two or one intersections, genus 0, and multiplicity $<$ sum of neighbours.
}

\begin{theorem}
\label{minmain}
Let $C: f=0$ be a $\Dv$-regular curve, with $f=\sum a_{ij}x^i y^j$.
Its \mrnc{} model $\cC_\Delta^{\min}/O_K$ 
has the following special fibre $\cC_k^{\min}/k$.
%
\begin{enumerate}
\item
Let $f_{nr}=\sum' a_{ij}x^i y^j$, the sum taken over those $(i,j)$ that lie in the closure
of some non-removable $v$-face of $\Delta$. Then $f_{nr}$ is $\Dv$-regular.
\item
Let $\cC$ be the model of $f_{nr}=0$ of \S\ref{sProof} and Theorem \ref{mainthm1},
and $\cC_k$ its special fibre.
The principal $v$-faces of $\Delta$ are in 1-1 correspondence 
with principal components of $\cC_k$ in the sense of Xiao \cite{Xiao} 
(positive genus or meeting the rest of $\cC_k$ in $\ge 3$ points).
\item
Starting with $\cC_k$, repeatedly contract $\P^1$s of self-intersection $-1$ 
in the chains between two principal $v$-faces\footnote{if there are no contractible faces,
and all chains in \ref{mainthm1}(2) are chosen to be minimal (cf. \ref{sternbrocot}),
then $\cC_k$ is already minimal and there is nothing to contract}. This gives $\cC_k^{\min}/k$.
%
\end{enumerate}

%
%
%
\marginpar{Move to front before $f_m$, as a separate statement. E.g. a component of 
$\cC$ is principal iff it comes from a principal face}
\marginpar{And Halle}
\marginpar{+Their singularities are almost always ignorable. 
  Alternatively, they are usually outer regular. But not always, cf.\\
  $(y\!-\!x\!-\!1)(2y\!-\!x\!-\!1)=px^6$}
\end{theorem}

\begin{proof}
Follows from Proposition \ref{propremcon} and Remark \ref{sternbrocot}.
\end{proof}

\endsection
\section{Reduction and inertia}
\label{sRed}

Thoroughout this section, $C$ is a $\Dv$-regular curve, and $\cC_k$ the special fibre of the
\rnc{} model $\cC_\Delta$ of \S\ref{sProof} and Theorem \ref{mainthm1}.
Write $\Gamma(\cC_k)$ for its \emph{dual graph}: it has geometric components as vertices, 
with an edge for every double point. Viewing it as a topological space, we compute its homology:


\marginpar{Recall that a curve acquires semistable reduction over a finite extension $K'/K$ 
(Deligne-Mumford), characterised by the fact that $I_{K'}$ acts unipotently on $H^1(C)$.
We say that $C/K$ is \emph{tame} if $I_K$ acts through a tame quotient, equivalently such an
extension is tamely ramified.}

\marginpar{By Saito's criterion \cite[Thm 3]{Saito}, this is equivalent to 
the minimal n.c.d. model satisfy his condition (*)}

\begin{theorem}[Homology of the dual graph]
\label{dualgraph}
Suppose $C$ is $\Dv$-regular. Then $H_1(\Gamma(\cC_k),\C)$ has dimension $|\Delta(\Z)^L_{\Z}|$, and
$$
   H_1(\Gamma(\cC_k),\Z) \>\>\iso\>\> \ker\Bigl(\bigoplus\subsmalltext4{inner}{$v$-edges $L$} \Z[X_L(\bar k)] 
     \overarrow{\phi} \Z[\text{$v$-faces}]
   \Bigr)
$$ 
as $G_k$-modules. Here $G_k$ acts naturally on $\Z[X_L(\bar k)]$,
trivially on $\Z[\text{$v$-faces}]$, and for $L\subset \partial F_1, \partial F_2$ 
and every $r\in X_L(\bar k)$, we set $\phi(r)=F_1-F_2$.
\end{theorem}

\begin{proof}
For every connected directed graph $\Gamma$, there is an $\Aut\Gamma$-equivariant isomorphism
\marginpar{Needs an explanation, as $\Aut\Gamma$ does not preserve the direction 
  of the edges, i.e. it should be viewed as undirected graph automorphisms, and the edges
  are signed? Or maybe forget it, our $\Aut\Gamma$ preserves directions}
$H_1(\Gamma,\Z) \iso \ker\phi$, where 
$$
  \phi\colon \Z^{\text{edges}} \lar \Z^{\text{vertices}}, \qquad \phi(v_1\!\to\!v_2)=v_1-v_2
$$
is as above.
Now take $\Gamma=\Gamma(\cC_k)$, obtained from Theorem \ref{mainthm1}; note that it is connected.
The choice of the ordering $F_1, F_2$ of $v$-faces whose boundary
is a given $v$-edge $L$ (already implicit in the definition of $\phi$ in the statement) 
makes $\Gamma$ directed. Now modify $\Gamma$, by viewing every chain of $\P^1$s as one component 
(i.e. edge), and contracting all outer chains.
This produces a new graph $\Gamma'$ with $H_1(\Gamma,\Z)=H_1(\Gamma',\Z)$. 

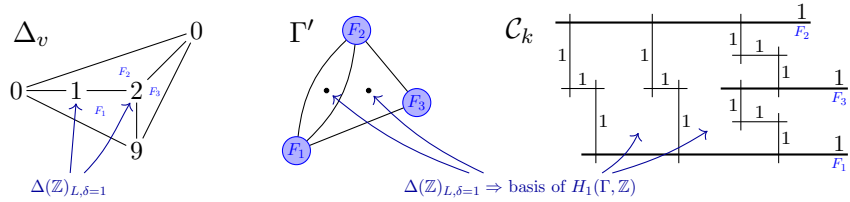
\begin{center}
\homologypic
\end{center}

\noindent
We find a planar realisation of $\Gamma'$ with 2-dimensional faces in 1-1 correspondence%
\marginpar{Explain $G_k$-action}
with elements of $|\Delta(\Z)^L_{\Z}|$; these give a $\Z$-basis of homology.
The statement on the $G_k$-action follows from Theorem \ref{mainthm1}.
\marginpar{Note that these basis elements are not permuted by $G_k$ in general, and in any case there
is no natural $G_k$-action on $\Delta(\Z)^L_{\Z}$.}
\end{proof}

Next, recall that $\Pic^0 \cC_k$ is an algebraic group over $k$, an extension 
of an abelian variety (the \emph{abelian part}) by a linear group; the latter is an extension
of a unipotent group by a torus (the \emph{toric part}) \cite[\S9.2]{BLR}.

\marginpar{Torus is determined by its character group}
\marginpar{$*$ Need $k$ perfect?}

\begin{corollary}[Special fibre]
\label{corabtor}
Suppose $C/K$ is $\Dv$-regular, and $k$ is perfect.
\begin{enumerate}
\item
The abelian part of $\Pic^0 \cC_k$ is isomorphic to 
$\bigoplus_{\text{$v$-faces $F$}}\Pic^0 \bar X_F$,
and has dimension $|\Delta(\Z)^F_{\Z}|$.
\item 
The toric part of $\Pic^0 \cC_k$ has dimension $|\Delta(\Z)^L_{\Z}|$
\marginpar{not $\oplus\triv^m$ but $/\Z^m$, because it is not true integrally?}
and character group $H_1(\Gamma(\cC_k),\Z)$ given by Theorem \ref{dualgraph}.
\item
Let $l\ne \vchar k$, and denote $H^i(-)=H^i(-_{\bar k},\Q_l)$. 
\marginpar{$V_l=H^1\tensor\Q_l(1)$, state somewhere; or formulate everything in terms
  of one or the other?}
We have
$$
  H^1(\cC_k)\ominus\triv\>\>\>\iso\>\> 
    \bigoplus_{\text{\scalebox{0.9}{$v$-faces $F$}}} (H^1(\bar X_F)\ominus\triv)
      \>\>\>\oplus\>
    \bigoplus\subsmalltext4{inner}{$v$-edges $L$} H^0(\bar X_L).
$$
\end{enumerate}
\end{corollary}

\begin{proof}
(1), (2)
Combine Theorems \ref{mainthm1}, \ref{dualgraph} with \cite[9.2.5/8]{BLR}.
(3) 
The image of $\phi$ of Theorem \ref{dualgraph} is the degree 0 part of $\Z[\text{$v$-faces}]$,
as $\Gamma=\Gamma(\cC_k)$ is connected. Tensoring with $\Q_l$, we find that the toric part is
$$
  H^1(\Gamma,\Z)\tensor\Q_l \iso H_1(\Gamma,\Z)\tensor\Q_l
   \iso \bigoplus_L H^0(\bar X_L)
     \ominus \Z[\text{$v$-faces}] \oplus \triv.
$$
Moving $\Z[\text{$v$-faces}]$ into the abelian part gives the claim.
\end{proof}


\begin{remark}
\label{locdata}
When $K$ is local and $C/K$ is semistable, this also determines
\begin{itemize}
\item 
the \emph{conductor exponent} ${\mathfrak f}_{C/K}=|\Delta(\Z)^L_{\Z}|$ (the toric dimension),~and
\item
the \emph{local root number} $w_{C/K}=\>$multiplicity 
of $\triv$ in $H_1(\Gamma,\C)$, see e.g. \cite[Prop. 3.23]{tamroot}.
\end{itemize}
\end{remark}

\begin{theorem}[Inertia action]
\label{tamecoh}
Suppose $k$ is perfect, 
$C/K$ is a $\Dv$-regular curve, and $l\ne p$.
For $P\in \Delta(\Z)_{\Z_p}$ define a tame character 
$$
  \chiP: I_K\lar\text{\{roots of unity\}}, \quad\qquad \sigma\longmapsto \sigma(\pi^{v(P)})/\pi^{v(P)}.
$$
Let $V^{\ab}_{\tame}$, $V^{\toric}_{\tame}$ be the unique continuous $l$-adic representations of $I_K$
that decompose over $\bar\Q_l$ as
$$
  V_{\tame}^{\ab}\>\> \iso_{\scriptscriptstyle\bar\Q_l}\!\!
  \bigoplus_{P\in \Delta(\Z)^F_{\Z_p}} \!\!\!\! (\chiP\oplus\chiPinv), \qquad
  V_{\tame}^{\toric}\>\> \iso_{\scriptscriptstyle\bar\Q_l}\!\!
  \bigoplus_{P\in \Delta(\Z)^L_{\Z_p}} \!\!\!\! \chiP.
$$
%
%
Then there are isomorphisms of $I_K$-modules, 
\marginpar{$\Sp_2$ needs to be defined. Note that for $V_l$ it may be only used 
  for inertia, otherwise twist it by 1}
$$
  \H(C_{\bar K},\Q_l)^{\Iwild}  \>\>\iso\>\> 
  (V_l \Jac C)^{\Iwild}\>\>\iso\>\> 
  V_{\tame}^{\ab} \>\>\oplus\>\> V_{\tame}^{\toric}\!\tensor\!\Sp_2.
$$
In particular, $\Jac C$ is wildly ramified $\iff$ $\Delta(\Z)_{\Z_p}\subsetneq \Delta(\Z)$. 
\end{theorem}

\begin{proof}
Write $J=\Jac C$. 
The statement only concerns inertia, so we may assume $k=\bar k$ and that $K$ 
is complete. The first isomorphism is standard.

Now $V_{\tame}^{\ab}$, $V_{\tame}^{\toric}$ are $I_K/\Iwild$-representations that 
factor through some cyclic group $C_d$, $p\nmid d$. They are rational (realisable over $\Q$)
by Lemma \ref{lemdelta} (2), and uniquely determined by the dimensions of their invariants 
under subgroups: for $H<C_d$ of index $e$, 
\begin{equation}\label{hinveq}
  \dim (V^{\ab}_{\tame})^H    = 2|\Delta(\Z)^F_{\oneoverZ{e}}|, \qquad  
  \dim (V^{\toric}_{\tame})^H = |\Delta(\Z)^L_{\oneoverZ{e}}|.
\end{equation}
On the other side, there is a canonical filtration 
of $I_K$-modules \cite[IX]{SGA7I}
$$
  (V_l J)^{\toric}\subseteq (V_l J)^{I_K} \subseteq V_l J
$$
that gives a decomposition
$$
  V_l J \>\iso\> (V_l J)^{\ab} \>\oplus\>  (V_l J)^{\toric}\tensor\Sp_2.
$$
Decompose the wild inertia invariants accordingly,
$$
  (V_l J)^{\Iwild} \>\iso\> (V_l J)^{\ab}_{\tame} \>\oplus\>  (V_l J)^{\toric}_{\tame}\tensor\Sp_2.
$$
The $I_K/\Iwild$-representations $(V_l J)^{\ab}_{\tame}$, $(V_l J)^{\toric}_{\tame}$ 
also factor through finite (cyclic) groups by the Grothendieck monodromy theorem.
They are rational, as their characters are independent of $l$, and a representation of cyclic
group with a rational character is rational. 
We need to show that they are $\iso V^{\ab}_{\tame}$ and $\iso V^{\toric}_{\tame}$. 
First note that the inertia invariants match:
$$
  \dim (V_l J)^{I_K} = \dim V_l \Pic^0\cC_k = 2|\Delta(\Z)^F_{\Z}| + |\Delta(\Z)^L_{\Z}|,
$$
with the first equality by Theorem \ref{bw2}, and the second by Corollary \ref{corabtor} plus the 
\marginpar{and \eqref{piciso}?}
fact that the unipotent part of $\Pic^0\cC_k$ has no $l$-power torsion. 
By Deligne's argument on Frobenius weights \cite[I,\,\S6]{SGA7I}, `the abelian and toric 
parts do not mix', so $(V_l J)^{\ab}_{\tame}$ reduces to the abelian part, and 
$(V_l J)^{\toric}_{\tame}$ to the toric part of $\Pic^0\cC_k$. In other words,
$$
  \dim (V_l J)^{\ab}_{\tame} = 2|\Delta(\Z)^F_{\Z}|, \qquad  
  \dim (V_l J)^{\toric}_{\tame} = |\Delta(\Z)^L_{\Z}|.
$$
The same argument applies over the (unique) tame extension of $K$ of degree $e$ for every $e\ge 1$,
because $C$ stays $\Dv$-regular in finite tame extensions by Remark \ref{dvbasechange},
and the function $v$ on $\Delta$ gets multiplied by $e$.
This proves the claimed isomorphism. Finally, 
$$
  \dim V_l J=2\dim J=2g(C)=2|\Delta(\Z)| = 2|\Delta(\Z)^F_{\Z_p}|+
    2|\Delta(\Z)^L_{\Z_p}| + 2|\Delta(\Z)^L_{\notin\Z_p}|.
$$
Comparing the dimensions with \eqref{hinveq} for $H=\{1\}$, we see
that the last term is non-zero if and only if $V_l J\ne (V_l J)^{\Iwild}$.
\end{proof}

\begin{example}
The curve from Example \ref{exc1} has wild reduction 
when $\vchar k$ is 2 or 3, and tame otherwise. In the tame case, $C$ becomes semistable%
\marginpar{Standard: see e.g. Liu, last section}
over any extension $K'/K$ of ramification degree 6. Let $\chi\colon I_K/\Iwild\to\bar\Q_l^\times$
be an order 6 character. By Theorem \ref{tamecoh}, as an $I_K$-module ($\vchar k\ne 2,3$),
$$
  \H(C_{\bar K},\Q_l)\>\>\iso\>\>
    \bigl( \chi\oplus\chi^{-1}\oplus\chi^2\oplus\chi^{-2}\bigr) 
    \>\oplus\>
    \bigl(\chi^2\oplus\chi^{-2}\bigr) 
    \tensor\Sp_2.
$$
\end{example}

\begin{theorem}[Reduction]
\label{redcond}
Suppose $C/K$ is $\Dv$-regular, of genus$\,\ge\!1$.~
Then

\marginpar{Or~from~Thm~\ref{glueing}?}
\marginpar{$K$ local?\\$g\ge 1$ in the definition of semistability/tameness?}
\marginpar{Add statement about gcd of multiplicities of geometric components}
\marginpar{Perhaps completely into intro?}

\begin{itemize}
\item[(1)]
$C$ is $\Dv$-regular over every finite tame extension $K'/K$.
\item[(2)]
$C$ has good reduction $\iff$ $\Delta(\Z)\!=\!F(\Z)$ for some $v$-face $F$ with $\delta_F\!=\!1$.
\item[(3)]
$C$ is semistable $\iff$ every principal $v$-face $F$ has $\delta_F=1$.
\item[(4)]
$C$ is tame $\iff$ every principal $v$-face $F$ has $p\nmid\delta_F$ ($p=\vchar k$).\\
In this case, $C$
is $\Dv$-regular over every finite extension $K'/K$.
\end{itemize}
Suppose $k$ is perfect, and let $J$ be the Jacobian of $C$. Then
\begin{itemize}[resume]
\item[(5)]
$J$ has good reduction  $\iff$  $\Delta(\Z)=\Delta(\Z)^F_{\Z}$.
\item[(6)]
$J$ is semistable  $\iff$  $\Delta(\Z)=\Delta(\Z)_{\Z}$.
\marginpar{semistable = has semiabelian reduction}
\item[(7)]
$J$ is tame  $\iff$     $\Delta(\Z)=\Delta(\Z)_{\Z_p}$. In this case,\\
$J$ has potentially good reduction  $\iff$  $\Delta(\Z)^L=\emptyset$, and\\
$J$ has potentially totally toric reduction  $\iff$  $\Delta(\Z)^F=\emptyset$.
\end{itemize}
%
%
%
\end{theorem}

\begin{proof}
Let $\cC/O_K$ be the \mrnc{} model of $C/K$ from Theorems \ref{mainthm1}, \ref{minmain},
with special fibre~$\cC_k$. 

(1), (4, second claim) This is proved in Remark \ref{dvbasechange}. 

(3, $\Leftarrow$) For every principal $v$-face $F$, we have $\delta_L=1$ for all
$L\subset\partial F$ (inner or outer) since $\delta_F=1$. 
It follows from \ref{mainthm1}, \ref{minmain} and \ref{sternbrocot} that $\cC$ is a semistable
model of $C$.

(3, $\Rightarrow$) If $\delta_F>1$ for some principal $v$-face, then 
$\cC$ is not semistable as it has a multiple fibre. By the uniqueness of the \mrnc{} model,
$C$ is not semistable.

(2, $\Leftarrow$)
All faces but $F$ are removable and $\cC_k$ has one smooth component, of multiplicity 1. 
So $C$ has good reduction.

(2, $\Rightarrow$) Conversely, the \mrnc{} model has one component, and so 
(i) there is only one principal face $F$, (ii) $\delta_F=1$ as $C$ is semistable, and 
(iii) the dual graph $\Gamma$ of $\cC_k$ is a tree by \ref{dualgraph}, so $\Delta(\Z)^L=\emptyset$.

(6) For an abelian variety $A/K$ semistable reduction 
is equivalent to inertia $I_K$ acting unipotently on $V_l A$ 
\cite[IX, 3.5/3.8]{SGA7I}, \cite[7.4.6]{BLR},
so the claim follows from \ref{tamecoh}. 

(5) Semistable reduction is good if and only if the toric part of $\Pic^0\cC_k$ is trivial, so 
this follows from \ref{corabtor}. 
\marginpar{Alternatively, apply \ref{tamecoh} and N\'eron-Ogg-Shafarevich.}

(4, first claim), (7). By (1), $C/K$ remains $\Dv$-regular in tame extensions $K'/K$, and the function
$v: \Delta\to \R$ gets multiplied by $e_{K'/K}$. So (3), (6) and \ref{corabtor}(1,2) give the claims.
\end{proof}

\begin{remark}
Let $C$ be a $\Dv$-regular curve, and suppose $k$ is perfect. 
Then \ref{redcond} and \ref{propdefic} show that the following conditions 
are equivalent:
\begin{itemize}
\item[(1a)]
$\Jac C$ is semistable (equivalently, $I_K$ acts unipotently on $V_l\Jac C$),
\item[(1b)]
$C$ is either semistable of genus $\ge 1$, or a genus 1 curve with $C(K)\ne\emptyset$ and 
special fibre consisting of one non-reduced genus 1 component,
\end{itemize}
and the following are equivalent as well:
\begin{itemize}
\item[(2a)]
$C$ has tame reduction,
\item[(2b)]
Principal components of the special fibre of the some \rnc{} model 
have multiplicities prime to $p$.
\end{itemize}
Thus, \ref{tamecoh} extends theorems of Saito \cite[3,\,3.8]{Saito} and Halle \cite[7.1]{Hal} 
for $\Dv$-regular curves, describing explicitly the wild inertia invariants.
\end{remark}

\marginpar{Some examples}

\endsection
\section{\'Etale cohomology over local fields}
\label{sEtale}

\marginpar{Note on motivic decomposition in the 0-dimensional case, both tame and wild.
To mollify the point that the model depends on the defining equation, we show 
that the faces are reflected intrinsically in this decomposition}

\marginpar{Incidentally, can this be used in the 2IV-3 reduction types 
to show that they do not have a $\Dv$-regular model? Inertia permutes principal components
over the field of semistable reduction}

Suppose $K$ is complete and $k$ is finite, so $K$ is a local field. 
We write $\Iwild\normal I_K\normal G_K$ for the wild inertia and inertia.
Recall that a choice of a \emph{Frobenius element} $\Phi\in G_K$ (acting as $x\mapsto x^{|k|}$ on $\bar k$)
identifies $G_K=I_K\rtimes\hat\Z$, and the \emph{Weil group} $W_K=I_K\rtimes\Z$, 
with $\hat\Z, \Z$ generated by $\Phi$. 
Fix a prime $l\ne \vchar k$; for a scheme $Y/F$ ($F\in\{k,K\}$) write $H^i(Y)=\H[i](Y_{\bar F},\Q_l)$.

We note that \ref{corabtor} gives
the \emph{local polynomial} $\det (1\!-\!\Phi^{-1} T\,|\,H^1(C)^{I_K}$),
\begin{equation}\label{locpoly}
  L(T) = \det (1\!-\!\Phi^{-1} T\,|\, H_1(\Gamma(\cC_k),\Z))
         \prod_{\text{$v$-faces $F$}} 
         \det (1\!-\!\Phi^{-1} T\,|\, \Pic^0 \bar X_F).
\end{equation}
The first factor has roots of absolute value 1.
The second is a product of $\zeta$-functions of $\bar X_F$ with trivial factors $(1\!-\!T)(1\!-\!|k|T)$
removed, and has roots of absolute value $|k|^{-1/2}$.

\def\pip{{\pi_{\scriptscriptstyle\!L}\!}}

%
%


%

\marginpar{Quote some example from before}

\comment
\bigskip

$$
  X_F^u: \quad  \sum_{i,j} \bar a_{ij} u^{v_F(i,j)} x^i y^j = 0, \quad\text{where}\quad
  \bar a_P = \leftchoice 0{\text{if }P\notin\Lambda_F,}{\frac{a_P}{\pi^{v_F(P)}}\mod\pi}{\text{if }P\in\Lambda_F.}
$$

$$
  X_F^u: \quad  \sum_{(i,j)\in\Z^2} \bar a_{\phi(i,j)}^u x^i y^j =0
  ; \qquad  \bar a_{P}^u=a_{P} \Bigl(\frac u\pi\Bigr)^{v_F(P)}\mod\pi.
$$

\begin{lemma}
\label{tamechange}
Let $K'/K$ be a field extension, $C\!\subset\!\G_{m,K}$ a curve, and $F\!\subset\!\Delta_v$ a $v$-edge or a $v$-face.
\begin{enumerate}
\item
Precise base change formula.
\item 
If $K'/K$ is unramified, or finite and tamely ramified, or finite and $\vchar k\nmid \Delta_F$, then
\begin{center}
$\bar f_F=0$ is smooth \qquad $\iff$ \qquad $\bar f_{K',F}=0$ is smooth.
\end{center}
\end{enumerate}
In particular, if $K'/K$ is unramified, or finite and tamely ramified, or finite 
and $\vchar k\nmid \Delta_F$ for every $v$-face $F$ of $\Delta$, then 
\marginpar{Outer regularity as well?}
\marginpar{$f\to X$?}
\begin{center}
$C/K$ is $\Dv$-regular \qquad $\iff$ \qquad $C/K'$ is $\Dv$-regular.
\end{center}
\end{lemma}

\begin{proof}
This is clear for unramified extensions, as smoothness commutes with base change.
Suppose $K'/K$ is finite, tame, and totally ramified. 
We may assume $K$ is complete.
Changing $\pi$ if necessary%
\footnote{pick a uniformiser $\eta$ of $K'$, write
$\eta^e=\pi+O(\eta^{e+1})$, and use that $\eta^e/\pi$ is an $e$th power.}%
, write $K'=K(\sqrt[e]{\pi})$ with $\vchar k\nmid e$. 
Up to $\GL_2(\Z)$-transformations, over $K'$ the defining equations
$\bar f_L(x)=0, \bar f_F(x,y)=0$ from $v$-edges/faces change to 
$\bar f_L(x^d)=0, \bar f_F(x^d,y)=0$ with $d=\gcd(e,\delta_L)$ and $d=\gcd(e,\delta_F)$.
These are still smooth as $\vchar k\nmid d$.
\end{proof}

\begin{lemma}
\label{tamechange2}
$F$ $v$-edge/face, $l|\delta_F$, $f_F^{(l)}$ restriction of $f|_F$ to the index $l$ sublattice.
Then $\bar f_F^{(l)}=0$ and $\bar f|_F=0$ define the same schemes over prime to $l$ extensions.
\end{lemma}
\endcomment


\marginpar{Need connected $v$-faces. Automatic?}

\begin{theorem}[Tame decomposition]
\label{tamedec}
Let $C/K$ be a $\Dv$-regular curve. 
For $v$-edges/faces $\lambda$ of $\Delta$, define $O_K$-schemes
\smash{$X_\lambda^{\tame}\colon f|_{\bar \lambda(\Z)_{\Z_p}} = 0.$} 
Then
\marginpar{When $p\nmid\delta_F$ this is just $X_F$ (cf. Definition \ref{defred}).}
\marginpar{=maximal tame quotient}
\marginpar{Incidentally, proves independence of $l$ (in the tame case), which might be 
interesting in higher dimensions}
\marginpar{Remark on Artin-Schreier at least in the case of hyperelliptic curves}
$$
  H^1(C)^{\Iwild}\ominus\triv\>\>\>\iso\>\> 
    \bigoplus\subsmalltext4{$F$}{$v$-faces} (H^1(X_F^{\tame})\ominus\triv)
      \>\>\>\oplus\>
    \bigoplus\subsmalltext4{$L$ inner}{$v$-edges} (H^0(X_L^{\tame})\tensor\Sp_2).
$$
\end{theorem}

\begin{proof}
Choose an embedding $\bar \Q_l\injects \C$ and view LHS and RHS as complex Weil-Deligne representations
of $W_K$.
From \eqref{locpoly}, \ref{corabtor} (3) and \eqref{eqininv}, we find that
${\rm LHS}^{I_K}\iso{\rm RHS}^{I_K}$ as $G_k$-modules. Generally, if $K'/K$ 
is a finite tame extension, then ${\rm LHS}^{I_{K'}}\iso{\rm RHS}^{I_{K'}}$ as 
$G_{K'}/I_{K'}$-modules, using the same argument as in Theorem \ref{tamecoh} and that
$\Dv$-regularity is stable in tame extensions. 
As ${\rm LHS}$ and ${\rm RHS}$ are tame Frobenius-semisimple Weil representations 
that have the same local polynomial in all finite tame extensions of $K$, they are 
isomorphic, by \cite[Thm 1.1]{weil}. (This is shown in \cite{weil} without the two words `tame', 
but the tame case is an easy consequence.) Finally, LHS and RHS are Weil-Deligne representations
of the form Weil$\>\oplus\>$Weil$\tensor\!\Sp_2$, so their semisimplifications determine the 
representations themselves, since Frobenius has different weights on the two constituents.
\end{proof}

\marginpar{Even for singular models understand a piece of \'etale cohomology?}


\begin{theorem}[One tame face] 
\label{onetame}
Suppose $C: \sum a_{ij}x^iy^j=0$ is $\Dv$-regular, $K$ is local with $k=\F_q$, 
$\Dv$ has only one $v$-face $F$, and $p\nmid \delta_F=e$.
Then
\begin{enumerate}
\item
$C$ has good reduction over $K(\sqrt[e]{\pi})$.
\item
The action $G_K\acts H^1(C)$ factors through $G\!=\!\Gal(K^{nr}(\sqrt[e]{\pi})/K)$. 
\marginpar{$\sqrt[e]{\pi}\to\pi^{1\!/\!e}$? No?}
It is faithful on $G$, and 
is uniquely determined by (3) and (4) below.
\marginpar{In fact (4) is enough}
\end{enumerate}
For $u\in\bar\F_q^\times$ define a complete smooth curve $\bar C_u/\F_q(u)$ by an affine equation
\marginpar{All twists of one another}
$$
  C_u\colon \> \sum_{i,j} \bar a_{ij} u^{v(i,j)} x^i y^j = 0, \quad\text{where}\quad
  \bar a_P = \leftchoice 0{\text{if }v(P)\notin \Z}{\frac{a_P}{\pi^{v(P)}}\mod\pi}{\text{if }v(P)\in \Z.}
$$
Write $I\!\iso\!\tfrac{\Z}{e\Z}<G$ for the inertia group, and $\Phi\in G$ for the Frobenius element
of $K(\sqrt[e]{\pi})$.
\begin{enumerate}[resume]
\item 
For every 1-dimensional character $\chi$ of $I$ of order $n$, the multiplicity
of $\chi$ in $H^1(C)$ viewed as an $I$-module is 
$$
   \tfrac{2}{\phi(n)}
   \bigl|\bigl\{P\in\Delta(\Z)\bigm| v(P) \text{ has denominator $n$}\}\bigr|.
$$
\item
Every conjugacy class of the Weil group $W=I\rtimes \Z<G$ which 
is not in~$I$ contains either an element~$\sigma\Phi^d$ or $(\sigma\Phi^d)^{-1}$
with $d\!>\!0$, $\sigma\!\in\!I$ and 
$
  \frac{\sigma(\sqrt[e]{\pi})}{\sqrt[e]{\pi}} 
$
\marginpar{Check from HQ Frob vs Frob$^{-1}$}%
\marginpar{Check that deficiency is not an issue}%
reducing to \smash{$u^{\frac{q^d\!-\!1}e}$} for some $u\!\in\!\F_{q^d}^\times$.~We~have
$$
  \Tr\bigl(\,(\sigma\Phi^d)^{-1}\bigm|H^1(C)\,\bigr) = q^d + 1 - |\bar C_u(\F_{q^d})|.
$$
\end{enumerate}
\end{theorem}

\begin{proof}
(1) Recall from Remark \ref{dvbasechange}(3) that $C$ stays $\Dv$-regular in finite 
tame extensions of $K$. This applies to $K'=K(\sqrt[e]{\pi})$, where $\delta_F$ becomes $1$, 
and then $C/K'$ has good reduction by Theorem \ref{mainthm1}. 

(2) Theorem \ref{mainthm1} also shows that the reduction is bad over $K(\sqrt[d]{\pi})$ 
for $d<e$. By the N\'eron-Ogg-Shafarevich criterion for $\Jac C$ (see \cite[\S2]{ST}),
\marginpar{Don't use NOS, use smooth and proper base change lemma}
$G_K$ acts on $H^1(C)$ through $G$, and faithfully on $I$; the determinant of this action
is a power of the cyclotomic character, so it is faithful on all of $G$. Finally, the traces of 
$g\acts H^1(C)$ for $g\in G$ determine $H^1(C)$ (as it is semisimple), 
so it is enough to compute them for $g\in G\setminus I$ and to describe $H^1(C)|_I$. 

\def\overref#1#2{\>{\buildrel\ref{#1}\over#2}\>}
\def\overcite#1#2{\>{\buildrel\scalebox{0.8}{\hbox{\cite{#1}}}\over#2}\>}

(3) This is a special case of Theorem \ref{tamecoh} and the fact that 
$H^1(C)|_I$ is a rational representation.

(4) Let $\Psi\in W\setminus I$, and let $U=\langle\Psi\rangle<W$ be the subgroup it generates.
Then $\Psi$ has infinite order, $U\cap I=\{1\}$, and the image of $U$ in $W/I\iso\Z$ 
is of some finite index $d$. The closure of $U$ in $G_K$ cuts out a finite extension $K''/K$,
of ramification degree $e$, residue degree $d$, and $\Psi$ is a Frobenius element of $K''$
or its inverse. Inverting $\Psi$ if necessary, suppose from now on that it is the former. 
So $\Psi=\sigma\Phi^d$, for some $\sigma\in I$.

Let $K\subset K'\subset K''$ be the degree $d$ unramified extension of $K$.
\marginpar{Explain? Used to be in Lemma tamechange}
Pick uniformisers $\pi''$ of $K''$ and $\pi'$ of $K'$ with $(\pi'')^e=\pi'$.
Then $(\pi'')^e=\tilde u\pi$ with $\tilde u\in K'$ a unit,
whose reduction we call $u\in\F_{q^d}^\times$.
Thus, $K''=K(\zeta_{q^d},\sqrt[e]{\tilde u\pi})$.

Now, $\Psi$ fixes $\sqrt[e]{\tilde u\pi}$, and sends $\sqrt[e]{\tilde u}$ to 
$(\sqrt[e]{\tilde u})^{q^d}$ modulo higher order terms, whence
$\frac{\Psi(\sqrt[e]{\pi})}{\sqrt[e]{\pi}}$ reduces to {$u^{\frac{q^d\!-\!1}e}$}.
Therefore $\sigma=\Psi\Phi^{-d}$ acts on $\sqrt[e]{\pi}$ 
in the same way (as $\Phi$ fixes it). This proves the first assertion.

Finally, as in (1), $C/K''$ has good reduction and the reduced curve is exactly 
$\bar C_u/\F_{q^d}$. As $\Psi\acts H^1(C)$ and $\Frob\acts H^1(\bar C_u)$ 
act in the same way\marginpar{e.g. by \cite{BW} or appendix},
the Lefschetz trace formula gives the second claim in (4).
\end{proof}

\begin{example}
\label{extamecoh}
Consider the two curves 
$$
  \begin{array}{l@{\ \ \rm over\ }l@{,\ \ \ }llll}
  y^2=x^3+7^2 & K=\Q_7 & \pi=7 \cr
  y^2=x^3+5^2 & K=\Q_5 & \pi=5. \cr
  \end{array}
$$
Both are $\Dv$-regular, with one $v$-face $F$ that has $e=\delta_F=3$. 

Tables \ref{tabex7} and \ref{tabex5}
list, for various $u$ and $d$ the reduced curve $\bar C_u/\F_{q^d}$, its number of points
and the conclusion of Theorem \ref{onetame} (4) in that case. These determine 
(with an overkill) the action of $G=\frac{\Z}{3\Z}\rtimes\hat\Z$ on $H^1(C)$, shown 
at the bottom of the tables. In the first case, $I=\frac{\Z}{3\Z}$ commutes with $\hat\Z$ and
the action is diagonal; here $\tau\in I$ is chosen to be the generator that multiplies 
$\sqrt[3]{7}$ by a 3rd root of unity congruent to 2 mod 7. In the second case, $\hat\Z$ acts
on $I$ by inversion and $G$ is non-abelian; here $\tau\in I$ is any generator (the two are 
conjugate in $G$). In both tables, $\Phi$ is the Frobenius element of $K(\sqrt[3]{\pi})$, as
in the theorem.

\begin{table}[h]
$$
  \scalebox{0.8}{$
  \begin{array}{|@{\ \ }cc|ccr@{\>=\>}c@{\>=\>}r@{\ \ }|}
  \hline
    u & d & \bar C_u & |\bar C_u(\F_{q^d})| & \multicolumn{3}{c|}{\Tr} \TBcr
  \hline
    \pm1 & 1 & y^2 = x^3+1  & 12 & \Tr(\Phi^{-1})&\ \>7+1-12&-4  \Tcr
         & 2 &              & 48 & \Tr(\Phi^{-2})&49+1-48&2 \cr
    \pm2 & 1 & y^2 = x^3+4  & 3  & \Tr(\Phi^{-1}\tau)&7+1-3&5 \cr
         & 2 &              & 39 & \Tr(\Phi^{-2}\tau^2)&49+1-39&11 \cr
    \pm3 & 1 & y^2 = x^3+2  & 9  & \Tr(\Phi^{-1}\tau^2)&7+1-9&-1 \cr
         & 2 &              & 63 & \Tr(\Phi^{-2}\tau)&49+1-63&-13 \Bcr
  \hline
  \multicolumn{7}{|c|}{
      \Phi^{-1}:\smmatrix{-2\!+\!\sqrt{-3}}00{-2\!-\!\sqrt{-3}}   \qquad 
      \tau:\smmatrix{\frac{-1-\sqrt{-3}}2}00{\frac{-1+\sqrt{-3}}2} \qquad 
      \Phi\tau=\tau\Phi
    }\TBcropt{3.9}{2.8}
  \hline
  \end{array}$}
$$
\caption{$G\!=\!\langle\tau,\Phi\rangle\!=\!C_3\!\times\!\hat\Z \acts H^1(C)$ for $C/\Q_7: y^2=x^3+7^2$}
\label{tabex7}
\end{table}

\begin{table}[h]
$$
  \scalebox{0.8}{$
  \begin{array}{|@{\ \ }cc|ccr@{\>=\>}c@{\>=\>}r@{\ \ }|}
  \hline
    u & d & \bar C_u & |\bar C_u(\F_{q^d})| & \multicolumn{3}{c|}{\Tr} \TBcr
  \hline
    \in\F_5^\times & 1 & y^2 = x^3+u^2  & 6 & \Tr(\Phi^{-1})&\ \>5+1-6&0  \Tcr
         & & & \multicolumn{4}{r|}{\scalebox{0.85}{$(=\Tr(\Phi^{-1}\tau)=\Tr(\Phi^{-1}\tau^2),\text{ same class})$}}\Bcr
         & 2 &              & 36 & \Tr(\Phi^{-2})&25+1-36&-10 \cr
    \zeta_{24} & 2 & y^2 = x^3+\zeta_{12} & 21 & \Tr(\Phi^{-2}\tau^2)&25+1-21& 5 \cr
    \zeta_{12} & 2 & y^2 = x^3+\zeta_{6}  & 21 & \Tr(\Phi^{-2}\tau)&25+1-21& 5 \cr
  \hline
  \multicolumn{7}{|c|}{
      \Phi^{-1}:\smmatrix0{\sqrt{-5}}{\sqrt{-5}}0   \qquad 
      \tau:\smmatrix{\frac{-1-\sqrt{-3}}2}00{\frac{-1+\sqrt{-3}}2}\qquad 
      \Phi\tau\Phi^{-1}=\tau^{-1}
    }\TBcropt{3.9}{2.8}
  \hline
  \end{array}$}
$$
\caption{$G\!=\!\langle\tau,\Phi\rangle\!=\!C_3\!\rtimes\!\hat\Z \acts H^1(C)$ for $C/\Q_5: y^2=x^3+5^2$}
\label{tabex5}
\end{table}
\end{example}

\begin{example}
\label{eisen}
Take any tame `2-variable Eisenstein equation' over $\Q_p$,
$$
  C: y^n+x^m+p\sum_{i,j} a_{ij}x^iy^j=0, \qquad p\nmid m\,n\,a_{00},
$$
with $(i,j)$ confined to the triangle of exponents of $y^n$, $x^m$ and 1.
Then $C$ is $\Dv$-regular, with one $v$-face $F$. We find that 
\begin{itemize}
\item 
$C$ has good reduction over $K(\sqrt[e]{\pi})$, with $e=\delta_F=\lcm(m,n)$ (tame).
\item
$H^1(C)|_{I_{\Q_p}}$ does not depend on $C$. It factors through the unique $C_e$-quo\-tient
of $I_{\Q_p}$, and is the permutation representation (see Thm.~\ref{tamecoh})
$$
  H^1(C)|_{I_{\Q_p}} \>\iso\> (\Q_l[C_e/C_{e/m}]\ominus\Q_l)\tensor (\Q_l[C_e/C_{e/n}]\ominus\Q_l).
$$
\item
$H^1(C)$ depends only on $a_{00}$~mod~$p$ (see Thm.~\ref{onetame}).
\end{itemize}
\end{example}

\begin{remark}
\label{remkisin}
Generally, by a theorem of Kisin \cite{Kis}, $l$-adic representations are locally constant in 
$p$-adic families of varieties. Thus, \ref{tamecoh} and \ref{onetame} may be viewed 
as an explicit version of this statement for tame $\Dv$-regular curves.
\end{remark}

\begin{remark}
We end by noting that \ref{tamedec}, \ref{onetame} and \ref{eisen} all fail without the tameness
assumption. There is one Kodaira type for elliptic curves, namely
I$_n^*$, when the regular model does not determine whether the reduction
is potentially good or potentially multiplicative (see \cite{LorM}).
So Theorem \ref{tamecoh} (and its refinements \ref{tamedec}, \ref{onetame}) 
do not have an obvious analogue for the description of $I_K\acts H^1(C)$ in the wild case
just in terms of $\Dv$.

Similarly, the condition $\vchar k\nmid mn$ in Example \ref{eisen} is necessary: 
for instance, the two elliptic curves over $\Q$
$$
  y^2=x^3+3, \qquad y^2=x^3+3x+3
$$
have conductors $2^43^5$ and $2^43^313$, so their Galois representations at~3 
are certainly not the same (while the regular model is the same, of type~II.)
\marginpar{A non-conceptual reason is that in $4A^3+27B^2$ when $3|A, 3||B$ the minimal
  discriminant (and, by Ogg's formula, the conductor)
  is influenced by whether $9|A$, even though the Newton polygon does not change}
\marginpar{Similarly, $y^2=x^3+2$ has inertia $Q_8$ and $y^2=x^3+4x+2$ has inertia $C_2$ at 2;
  though all conductor exponents seem to be always 6}
\end{remark}

\marginpar{For additional examples see \cite{weil} and \cite{hq}.}

\endsection
\section{Differentials}
\label{sDiff}

Let $v: K^\times\surjects \Z$, $O_K$, $\pi$, $k$, $p=\vchar k$ be, as usual, a discretely valued field 
and the associated invariants. Let $C: f=0$ be a $\Dv$-regular curve;
we choose variables as in \S\ref{sBaker} so that $f'_y\ne 0$. Recall that \ref{thmbaker} (4) 
gives a basis of differentials for $C/K$. Here we aim to modify it to give a basis of the 
global sections of the
relative dualising sheaf for $\cC_\Delta/O_K$ (Theorem \ref{mainthmdiff}).

We refer to \S\ref{sProof} for the charts for the components $\bar X_F$, $X_L\times\Gamma_L$ 
of the special fibre of $\cC_\Delta$, defined in \ref{defXFXL}.

\begin{proposition}
\label{difforder}
Let $F$ be a $v$-face of $\Delta$, and 
\marginpar{... or comes from a chain of $\P^1$s corresponding to a $v$-edge}
$F^*: \Z^2 \to \Z$ be the unique affine function that equals $-\delta_F v$ on $\bar F(\Z)$. Then
\marginpar{In fact, the proof does not use that $C$ is $\Dv$-regular, only that 
$\bar X_F$ is reduced? Though need regularity of $\cC$ to talk about differentials; or just
$\cC$ normal?}
$$
  \begin{array}{llllll}
     \displaystyle \ord_{\bar X_F} x^i y^j   &=& F^*(i,j) \!-\! F^*(0,0),  \cr
     \displaystyle \ord_{\bar X_F} \omegapij &=& F^*(i,j) \!-\! 1.  \cr
  \end{array}
  \qquad \text{for $(i,j)\in\Z^2$}.
$$
\end{proposition}

\def\stepd#1#2{\refstepcounter{equation}\label{ssd:step#1}\theequation. \emph{#2.}}
\def\refstepd#1{\ref{ssd:step#1}}

\noindent
\emph{Proof.}
\stepd 1{Coordinate transformation}
Write $S=\Spec O_K$, and~let 
$$
  (X,Y,Z)=(x,y,\pi)\bullet M,  \quad 
  M=(m_{ij}), \quad  M^{-1}=(\mt ij)
$$
be a coordinate transformation as in \refstepr 1-4, associated to 
$F$ and some $v$-edge $L$ of $F$.
\marginpar{$F$ is both a $v$-face and a transformed equation}
\marginpar{We will show that ...}
Recall from the formula \eqref{Minveqn} for $M^{-1}\in\SL_3(\Z)$ that
$$
  \mt 13=k_{i+1}d_i-k_id_{i+1}, \qquad \mt 23=\delta_L d_{i+1},
    \qquad \mt 33 = \delta_L d_i.
$$
If $\vchar k=p>0$ and $p|\delta_L$, then $\gcd(\mt 13,\mt 23,\mt 33)=1$
forces $p\nmid \mt13$; if $p\nmid\delta_L$, 
$p\nmid d_i$, $p|\mt 23$, we choose $k_{i+1}$ (defined modulo $\delta$) so that $p\nmid \mt 13$. 
Thus, $(\mt 13,\mt 23)\ne (0,0)\in K^2$.
Equivalently, 
\begin{equation}\label{m1323}
  (m_{13},m_{23})\ne (0,0)\in K^2,
\end{equation}
since $M$ preserves $(0,0,1)\in K^3$ if and only if $M^{-1}$ does.

Also, the first two columns of $M$ span the plane orthogonal to $\ker F^*$, 
\marginpar{Subsection and equation should be the same counter}
so this plane is parallel to $(\mt 31,\mt 32,\mt 33)^\perp$. As $\delta_F=\mt 33$, we find
\begin{equation}\label{mt123}
  F^*(i,j) = \mt 31 i + \mt32 j + F^*(0,0).
\end{equation}
\stepd 2{Transformed equation}
Write 
$G(X,Y,Z) = f((X,Y,Z)\bullet M^{-1})$,\linebreak
$G_0(X,Y,Z) = Z^{F^*(0,0)} G(X,Y,Z)$ and
$H(X,Y,Z)=X^{\mt 13}Y^{\mt 23}Z^{\mt 33}$,
a transformed version of $f$ and $\pi$ to the $X,Y,Z$-chart. Then
$$
  U\colon\>\> G_0=H=0
$$
defines a complete intersection in $\A^3_S$, 
and restricted to $\A^3_S\setminus\{XY=0\}$ it gives an open subset $U: Z=0$ of 
$\bar X_F\subset\cC_k^{\red}$. By \cite[6/4.14]{Liu}, the sheaf $\omega_{\cC/S}$
is generated on $\bar X_F$ by $dZ/\scalebox{0.9}{$\smalldet{(G_0)'_X}{(G_0)'_Y}{H'_X}{H'_Y}$}$ 
if this determinant is non-zero.
We will show that it is indeed non-zero (\refstepd 3, \refstepd 4), and that
$$
  Z^{F^*(0,0)}
  dZ/\smalldet{(G_0)'_X}{(G_0)'_Y}{H'_X}{H'_Y} = \frac{\pi^{-1} X Y dZ}{\mt23 X G'_X - \mt13 Y G'_Y}
  \overequal{\raise 3pt\hbox{\scalebox{0.6}{\eqref{zZeq}}}}
  -\pi^{-1} X Y Z \frac{dx}{xy f'_y}.
$$
As $X$,$Y$ are units on $U$, and $Z$ vanishes to order 1, we get the claim for $(i,j)\!=\!(0,0)$.
\marginpar{check}
As $\ord_U x\!=\!\mt 31, \ord_U y\!=\!\mt 32$, the theorem follows from~\eqref{mt123}.

\noindent
\stepd 3{Relation $f'\leftrightarrow G'$}
As in Lemma \ref{diffGm}, from the chain rule we get
$$
  \scalebox{0.9}{$\vec{xf'_x}{yf'_y}{\pi f'_\pi} = \vec{XG'_X}{YG'_Y}{Z G'_Z}$} M^t
  \quad\Rightarrow\quad 
  \left\{
  \begin{array}{@{\>}l@{\>=\>}l}
    xf'_x & m_{11} XG'_X \!+\! m_{12} YG'_Y \!+\! m_{13} ZG'_Z,\\[2pt]
    yf'_y & m_{21} XG'_X \!+\! m_{22} YG'_Y \!+\! m_{23} ZG'_Z.
  \end{array}
  \right.
$$
\marginpar{Again, viewing $f$ as an equation in 3 variables. Make this global?}%
Therefore,
\begin{equation}\label{fGdiffrel}
\begin{array}{llllll}
  m_{23}xf'_x-m_{13}yf'_y = 
  (m_{23}m_{11}\!-\!m_{13}m_{21}) XG'_X +  \cr
    \qquad (m_{23}m_{12}\!-\!m_{13}m_{22}) YG'_Y + 0\cdot Z G'_Z =
  \mt23 XG'_X - \mt13 Y G'_Y.
\end{array}
\end{equation}

\noindent
\stepd 4{Non-vanishing}
We claim that $m_{23}xf'_x-m_{13}yf'_y\ne 0$ in $K(C)$. If not, there is a 
linear relation in $K[x,y]$, non-trivial by \eqref{m1323},
$$
  c_1 xf'_x + c_2 yf'_y + c_3 f = 0, \qquad c_i\in K.
$$
(The coefficient $c_3$ must be constant,
\marginpar{check}
and not a higher degree polynomial in $x$ and $y$ by degree considerations.) Then
$\sum a_{ij}(c_1 i+c_2 j + c_3) x^i y^j$ is identically zero, so all monomial exponents are 
in the kernel of a non-trivial linear form. But this contradicts $\vol(\Delta)>0$.

\noindent
\stepd 5{$dx/f'_y\leftrightarrow dZ/\det$}
Recall that $Z=x^{m_{13}}y^{m_{23}}\pi^{m_{33}}$. 
From the relations 
$\tfrac{dZ}{Z} = m_{13} \tfrac{dx}x + m_{23} \tfrac{dy}y$ and 
$df = f'_x dx+ f'_y dy =0$ in $\Omega_{C/K}$, we get 
$$
  \frac{dZ}{Z} = \bigl( \frac{m_{13}}x - \frac{m_{23}}y \frac{f'_x}{f'_y} \bigr ) dx 
    = \frac{dx}{xyf'_y} (m_{13}yf'_y - m_{23}xf'_x).
$$
Combined with \eqref{fGdiffrel} this yields
\begin{equation}\label{zZeq}
  \frac {dZ}{\mt23 X G'_X - \mt13 Y G'_Y} = -Z \frac{dx}{xyf'_y}.
\end{equation}
\qed

\begin{remark}
In the notation of 
Theorem \ref{thmbaker} and Proposition \ref{difforder},
$\pi^n\omega_P$ is regular on $\bar X_F$ if and only if $n\ge\lfloor v_F(P)\rfloor$,
and regular non-vanishing on $\bar X_F$ if and only if 
$v_F(P)=n+1-\frac{1}{\delta_F}$.
\end{remark}

\begin{theorem}
\label{mainthmdiff}
If $C$ is $\Dv$-regular, then the differentials 
\marginpar{Just $\cC=\cC_{\Delta}$ regular?}
\marginpar{Need much less? All $\bar X_L$, $\bar X_F$ reduced?}
\marginpar{Make notation official?}
$$
  \omega_{ij}^v=\pi^{\lfloor v(i,j)\rfloor} \omegaij, \qquad (i,j)\in\Delta(\Z)
$$ 
form an $O_K$-basis of global sections of the relative dualising sheaf
\marginpar{=canonical sheaf}
$\omega_{\scriptscriptstyle \cC_\Delta/O_K}$.
\end{theorem}

\begin{proof}
By Proposition \ref{difforder}, each $\pi^{\lfloor v(i,j)\rfloor} \omega_{ij}$ 
has order $\ge 0$ at every component of $\cC_k$. So they are global sections, linearly independent by 
Baker's theorem, and it remains to prove that the lattice they span is saturated in the global sections. 
Suppose not. Then there is a combination of the form
\begin{equation}\label{diffzerocomb}
  \frac 1{\pi}\sum_{(i,j)\in \Sigma} u_{ij} \omega_{ij}^v \qquad\qquad 
    (\emptyset\ne\Sigma\subset\Delta(\Z),\>\> u_{ij}\in O_K^\times),
\end{equation}
which is regular along every component. Pick $P\in\Sigma$ 
with $\epsilon=v(P)\!-\!\lfloor v(P)\rfloor$ maximal,
and a $v$-face $F$ with $P\in\bar F(\Z)$. Write
$$
  m = \ord_{\bar X_F}\frac 1{\pi}\,\pi^{\lfloor v(P)\rfloor} \omega_P; \qquad  
  m \overref={difforder} -\epsilon\delta_F-1<0.
$$
As \eqref{diffzerocomb} is supposed to be regular along $\bar X_F$, and all its terms 
have order~$\ge m$ on $\bar X_F$ by maximality of $\epsilon$ and lower convexity of $\Dv$,
those of order $m$ must cancel along $\bar X_F$. Let $\Sigma_m\subset \Sigma$ be their indices.

First, note that $\Sigma_m\subset\bar F(\Z)$. Indeed, else there is $P'\in\Sigma_m$ on some 
$v$-face $F'\ne F$ of $\Delta$ for which $v(P')\!-\!\lfloor v(P')\rfloor>\epsilon$, contradiction.
Now, from Remark \ref{difind} it follows that $\Sigma_m$ contains all vertices of $F$.
\marginpar{Expand}
This allows us to replace $F$ by any of the neighbouring $v$-faces, changing $P$ if necessary but
not changing $m$, and it also shows that $\bar F\cap\partial\Delta=\emptyset$ (as 
$\Sigma\cap\partial\Delta=\emptyset$). Proceeding inductively, we find that \eqref{diffzerocomb} 
cannot exist at all (cf. \cite{O'C}).
%
\end{proof}

\marginpar{
Easier to prove in semistable case, just $x\to \pi^m x, y\to \pi^n y$ as chart substitutions.
Incidentally, then they are also non-vanishing.
Assume $C$ is semistable. For one face this is a direct consequence of Baker, after rescaling.
For differentials, at least for integral faces, under $x=p^m X$, $y=p^n Y$, 
$
  \omega = p^{m(i-1)}X^{i-1}p^{n(j-1)}Y^{j-1}\frac{p^m dX}{p^{-n}\partial_yf} = p^{mi+nj} X^{i-1} Y^{j-1} \frac{dX}{\partial_yf}
$
and $mi+nj=v(i,j)$. Here $F(X,Y)=f(p^m X,p^n Y)=f(x,y)$.
}

\begin{example}
\label{exdeficient}
Let $C: y^2=\pi x^4+\pi^3$. If $\vchar k\ne 2$, then $C$ is $\Dv$-regular,
and its $\Delta_v$ and the special fibre of the \mrnc{} model are as follows:

\begin{center}
\EXDEFICIENT
\end{center}

It is a genus 1 curve with $C(K)=\emptyset$
by \ref{redpts}, so it is not an elliptic curve; it has bad reduction 
but its Jacobian has good reduction, by \ref{redcond}(2,5).
The differential $\pi dx/y$ spans the global sections of the relative dualising sheaf 
$\omega_{\cC_\Delta/O_K}$ by the above theorem, but it vanishes 
along the unique (reduced) component $\Gamma$ of the special fibre, by \ref{difforder}. Hence, 
$\omega_{\cC_\Delta/O_K}\>(=\!\!O(\Gamma))$ is not generated by global sections.
In contrast, $\omega_{\cC_\Delta/O_K}$ is trivial for genus 1 curves \emph{with a rational point},
see \cite[9.4,\,Exc.\,4.16]{Liu}.
\end{example}

\marginpar{General remarks on being generated by global sections?}

\begin{example}
\label{ex188}
Consider the genus 2 curve `188' from \cite{FLSSSW} at $p=2$,
$$
  C\colon y^2=x^5-x^4+x^3+x^2-2x+1.
$$
Letting $x\!\to\!\frac2{x-1},\>y\!\to\!\frac{y+1}{(x-1)^3}+\frac{x}{x-1}$, the equation becomes
$$
  y^2 + 2(x^3\!-\!2x^2\!+\!x\!+\!1)y + x(2x^2\!-\!5x\!+\!4)^2 = 0.
$$
This is $\Dv$-regular, with special fibre of the \mrnc{} model as follows:

\smallskip \qquad \OneEightEightCurve

\noindent
The differentials $dx/y$, $xdx/y$ form a basis of the relative dualising sheaf.
\end{example}

\endsection
\section{Example: Elliptic curves}
\label{sEll}

\def\Wei{{\text{\rm ($\cW$)}}}

\begin{table}[!htbp]
\input{ellc.inc} 
\vphantom{.}
\caption{Reduction types of elliptic curves (Theorem \ref{ellmain})}
\label{elltable}
\end{table}

As an application, we recover Tate's algorithm 
for elliptic curves\marginpar{Though not Ogg-Saito formula in the wild case},
and some immediate consequences for the Galois representation and the N\'eron 
component group.
All of this is well-known: see \cite[\S IV.9]{TaA,Sil2}+$\epsilon$\cite{tate} 
for Tate's algorithm
and the valuations of the coefficients of the minimal Weierstrass model, 
\cite{Kra,Roh,ST} for the Galois representation, and 
\cite[\S10.2,\,Exc.\,2.2]{Liu} for \rnc{} models of elliptic curves.

As always, we have $v: K^\times\surjects\Z$, $O_K$, $\pi$, $k$ and $p=\vchar k$. Recall that
an elliptic curve $E/K$ has a Weierstrass equation,
$$
  y^2 + a_1xy + a_3 y = x^3 + a_2 x^2 + a_4 x + a_6, \qquad a_i\in K,\>\Delta_E\ne 0.
  \eqno{(\cW)}
$$
Here $\Delta_E$ is the discriminant of $E$, defined via
\marginpar{$\Delta_E\to\Disc(E)$?}
$$
  b_2=a_1^2+4a_2,\quad 
  b_4=2a_4\!+\!a_1a_3,\quad 
  b_6=a_3^2\!+\!4a_6,
$$
$$  
  b_8=a_1^2a_6\!+\!4a_2a_6\!-\!a_1a_3a_4\!+\!a_2a_3^2\!-\!a_4^2,\quad 
  \Delta_E=-b_2^2b_8\!-\!8b_4^3\!-\!27b_6^2\!+\!9b_2b_4b_6.
$$
We write $f$ for LHS-RHS of \Wei, so that $f=0$ is the equation defining $E$.


\begin{theorem}
\label{ellmain}
Assume\footnote{Otherwise there are indeed 
more reduction types, classified in \cite{Szy}; the proof of \ref{ellmain} uses that a multiple root 
of a polynomial of degree $2$ or $3$ over $k$ is always $k$-rational.}
$k$ is perfect or $p>3$.
An elliptic curve $E/K$ has a $\Dv$-regular Weierstrass equation \Wei{}
with $\Delta_v$ as in Table \ref{elltable}, column 2,
up to removable faces. 
The Kodaira type of~$E$, special fibre of the \mrnc{} model, 
condition for $E$ to have tame reduction,
inertia action on $\H(E/\bar K,\Q_l)$ if $E$ is tame and $k$ is perfect, and the N\'eron component 
group $E(K)/E_0(K)$ are as in columns 1,3--6. 
%
%
\end{theorem}

\marginpar{The tweak is necessary because having minimal 
discriminant does not quite guarantee 
  that the model is $\Dv$-regular; e.g. $E: y^2=x^3-x$ is minimal at $p=2$, but 
  requires $x\mapsto x+1, y\mapsto y+x$ to get a $\Dv$-regular 
  model $y^2+2xy=x^3+2x^2+2x$ (of type III).
}

\marginpar{If $k$ is large enough, a shift can make it exactly like in the table, but
that is artificial.}


\marginpar{How about quartics?}

\begin{proof}
Start with a Weierstrass equation $(\cW)$. 
As $x=\pi^{-2} x', y=\pi^{-3} y'$ changes $a_i\mapsto \pi^i a_i$,
we can assume that all $v(a_i)\ge 0$. Write $P=(1,1)\in\Z^2$.

A transformation $x\mapsto x+r, y\mapsto y+sx+t$
with $r,s,t\in O_K$ keeps $\Delta_E$ unchanged, and $\Delta$ 
confined to the triangle (0,0)-(3,0)-(0,2).
As long as ($\cW$) is not $\Dv$-regular, we will see that there is such a transformation
that increases $v(P)$. Note that vertices of any 
$v$-face inside the triangle span an affine lattice of index $|12$ in $\Z^2$, so $12v(P)\in\Z$.
But $v(a_i)\ge i\,v(P)$ for all $i$, so $12v(P)\le v(\Delta_E)$ and the algorithm terminates.

There are four reasons why $f$ could be non-$\Dv$-regular:

\begin{enumerate}
\item 
$\Dv$ has $v$-edge $\lambda$ (left, (0,0)-(0,2)) with $\overline{f_\lambda}$ non-squarefree, and bounding a 
non-removable face.
\item
$\Dv$ has $v$-edge $\sigma$ (skew, (0,2)-(2,0)) with $\overline{f_\sigma}$ non-squarefree.
\item
$\Dv$ has $v$-edge $\beta$ (bottom, (0,0)-(2,0), (0,0)-(3,0) or (1,0)-(3,0)) with $\overline{f_\beta}$ 
non-squarefree, and bounding a non-removable face.
\item
$\Dv$ has a principal $v$-face $F$, with $P$ in its interior, and $\overline{f_F}=0$
is a geometrically singular subscheme of $\G_m^2$.
\end{enumerate}
In all four cases, the offending 1- or 2-face has denominator 1, for otherwise the corresponding 
reduction is linear and cannot define a singular scheme;
for the same reason, in (1) and (3) the face is assumed to be non-removable.
%

In case (1)-(3), the reduction $\overline{f_*}$ ($*\in\{\lambda,\sigma,\beta\}$) has a unique multiple 
root $\alpha\in k$, and a shift $y\mapsto y+\alpha\pi^d$, $y\mapsto y+\alpha\pi^d x$ or $x\mapsto x+\alpha\pi^d$ 
with $d$ defined by the slope of $*$ shifts the root to 0, breaking the 1-face $*$. Let us call this 
change of variables the `shift along $*$'. 

We refer to cases, i.e. table rows, by Kodaira types (first column). 

Step 1. 
%
%
(I$_n$)
If $P$ is a $v$-vertex of $\Delta$, then $v(1,0)\!>\!v(P), v(0,1)\!>\!v(P)$, 
and $x\mapsto x-\frac{a_3}{a_1}, y\mapsto y+\frac{a_4}{a_1}$ 
makes $a_3=a_4=0$; as $\Delta_E\ne 0$, we have $a_6\ne 0$, 
Now all $v$-faces are contractible, with linear reductions, and $f$ is always $\Dv$-regular.
The reduction type is I$_n$ (cf. Remark \ref{remprin}), split in this case.
\marginpar{Except we don't have a picture like that}

Step 2. If $\Delta_v$ has the $v$-edge $\sigma:$ (0,2)-(2,0), there are 3 possibilities.

If $\delta_\sigma=1$ and $\overline{f_\sigma}$ is not-squarefree, then a shift 
along $\sigma$ 
increases $v(P)$, and we go back to Step 1.

(I$_n$) If $\delta_\sigma=1$ and $\overline{f_\sigma}$ is squarefree, 
\marginpar{Except the face on the left can be a large triangle, and $\Dv$-regularity fails}
we are again in the I$_n$-case, possibly non-split
(with picture as in the table, up to possible removable 
faces). 

(I$_n^*$)
Otherwise $\delta_\sigma=2$, and $v(P)=\frac{m}{2}$, $v(2,0)=m$ for some odd $m$.
The only possibly singular faces are the $v$-edges $\lambda$ and $\beta$.
If one of them is, shift along it.
%
This increases the (smallest) slope of a face left of $P$. Since
$$  
  v(a_3),v(a_4)\ge t; v(a_6)\ge 2t \Rightarrow 
  v(b_4)\ge t; v(b_6),v(b_8)\ge 2t \Rightarrow v(\Delta_E)\ge 2t,
$$
this process terminates in a $\Dv$-regular model, of Type I$_n^*$.

Step 3. Finally suppose $P$ is neither $v$-vertex nor lies on a $v$-edge; in other words, it 
lies inside a unique principal $v$-face, say $F$. 

If $\{\bar f_F=0\}\subset\G_m^2$ is singular (Case (4)), it has arithmetic genus 1, implying $\delta_F=1$,
and the singular point is unique and $k$-rational. Use a shift in $x$ and $y$ to translate it to $(0,0)$.
If $v(P)$ increases or $P$ is not interior in a $v$-face any more, go back to Step 1.
Otherwise, the only possibly singular faces now are the $v$-edges $\lambda$ and $\beta$.
If one of them is, shift along it; this necessarily increases $v(P)$ and we go back to Step 1.
\marginpar{Check when $k$ is not perfect}

If we reached this stage, we now have a $\Dv$-regular model, with one non-removable
face. Depending on $v(P)=$integer+$0,\frac16,\frac14,\frac13,\frac12,\frac23,\frac34,\frac56$
(every face with vertices (3,0),(0,2) and somewhere else in $\Delta$ has $\delta_F|4$ or $|6$),
we get types I$_0$, II, III, IV, I$_0^*$, IV$^*$, III$^*$, II$^*$ after a rescaling.

Finally, the Galois action on the components of multiplicity 1 follows from Theorem \ref{mainthm1}(2),
and the `tame$\iff$' and `inertia if tame' columns in the table by Theorems 
\ref{redcond}(4) and \ref{tamecoh}%
\marginpar{Gives Tamagawa number, but not the Tamagawa group structure?}.
\end{proof}

\marginpar{Explain where the rest of the table comes from - Kraus + well-known, or 
  later sections}

\begin{remark}
One also recovers the behaviour of minimal models of elliptic curves in tame 
extensions, see \cite{tate} Thm 3 (1)-(3).
%
\marginpar{See \cite{tate} also for necessary and sufficient conditions for that reduction type
in terms of $\min v(a_i)/i$ and squarefreeness; 
or put in the table?}
\end{remark}

\marginpar{
Works quite well for hyperelliptic curves of arbitrary genus, except the need to handle different
singularities separately, and does not resolve ``multiple types'' such as 2IV-n. Note that,
as in Tate's algorithm, there is no difference between $p=2$ and $p$ odd.
}

\endsection
\section{Example: Fermat curves}
\label{sFermat}

Tate's algorithm relies on Weierstrass models having
a unique singularity that can be shifted to the origin. In higher genus, 
there may be multiple ones, and the algorithm `repeatedly shift to resolve a singularity'
of \ref{ellmain} often works but only \emph{locally}. 
Thus, \ref{mainthm2} may be applied for several choices of generators $x,y\in K(C)$ 
and the resulting schemes merged to a \rnc{} model.

For example, consider a genus 2 hyperelliptic curve
$$
  C/K\colon\> \pi Y^2=X^3(X-1)^2+\pi \qquad \qquad (\vchar k>3).
$$
The polynomial on the right
has a triple root $X\equiv 0$ and a double root $X\equiv 1$ mod $\pi$, so the equation 
is a model $\cC_{\rm sing}/O_K$ with two singular points 
\begin{center}
$P_1=(X,Y,\pi) \qquad\text{and}\qquad P_2=(X-1,Y,\pi)$.
\end{center}
Applying 
Theorem \ref{mainthm2} to $(x_1,y_1)=(X,Y)$ and $(x_2,y_2)=(X-1,Y)$
yields two schemes (with special fibres left and centre), both still singular:

\marginpar{ 
Need a theorem how to glue our models, at least in simplest necessary cases 
(multiple roots along a $v$-edge, translated separately to 0).
Every shift should be viewed as an isomorphism between open pieces of the two regular models
(open on the generic fibre, affine subset of a component on the special fibre).
Our explicit charts should give a formal way to glue different models together.
}

\EXGLUEING

\begin{center}
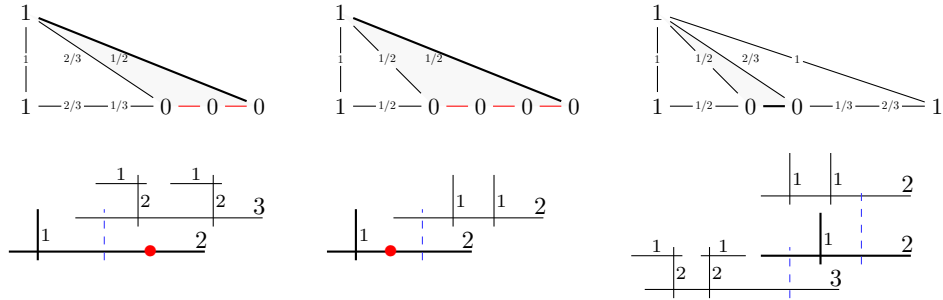
Shifts and glueing; red = singular point on the model
\end{center}

\noindent
The left one successfully resolves $P_1$ (dashed line and above) but not $P_2$ (red dot) and 
the centre one $P_2$ (dashed line and above) but not $P_1$ (red dot). 
They agree on two components --- thick one of multiplicity 1 coming from the thick $v$-edge of $\Delta$ 
above it, and a thick one of multiplicity 2 coming from the shaded $v$-face of $\Delta$ above it,
with $P_1, P_2$ removed. Glueing them together gives a \rnc{} model, bottom right.
Here, one \emph{could} resolve both $P_1$ and $P_2$ simultaneously by a M\"obius 
transformation that sends their $x$-coordinates to 0 and $\infty$ respectively (top right),
but that is generally not possible. 

\marginpar{$\Z$ or $\Z_p$?}
As an application, we construct regular models of Fermat curves 
$$
  x^p+y^p=1, \qquad\qquad p\ge 3
$$ 
over $\Z_p$ (and thus over its tame extensions as well, see Remark \ref{dvbasechange}). 
These are known over $\Z_p[\mu_p]$ \cite{McC,CM}, but seemingly not over $\Z_p$.
Let $x\mapsto x-y+1$, so that the new equation is
$$
  (x-y+1)^p + y^p = 1.
$$
The five monomials $x^p$, $px$, $-py$, $-py^{p-1}$, $-pxy^{p-1}$ determine
$\Delta$ and $\Dv$, as all the others have valuations $\ge 1$ and
lie inside the Newton polygon.

\smallskip
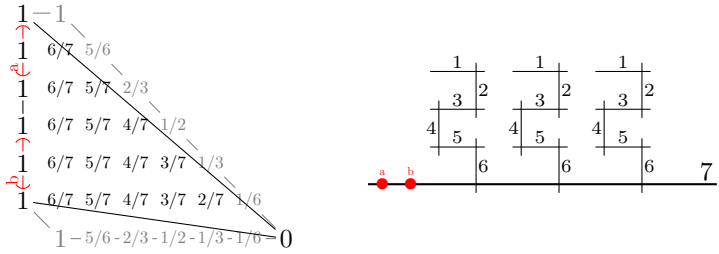
\quad\FERMATONE

\begin{center}
Example: $\Dv$ for $(x\!-\!y\!+\!1)^p + y^p = 1$ with $p=7$. 
\end{center}

\noindent
The $\Dv$-regularity conditions are automatic except at the leftmost $v$-edge $L$, where
the reduced equation is non-linear, namely
$$
  \bar f_L = \frac1p \bigl((1-y)^p + y^p - 1\bigr) \in \F_p[y].
$$
As shown in \cite[Lem. 3.1]{CM} (or \cite[p. 59]{McC}), considering the derivatives
$$
  \bar f_L' = -(1-y)^{p-1}+y^{p-1}, \qquad 
  \bar f_L'' = (p-1)(1-y)^{p-2}+(p-1)y^{p-2}
$$
and noting that the roots of $\bar f_L'$ are $2,3,...,p-1$, 
we see that all multiple roots of $\bar f_L$ are double and $\F_p$-rational%
\footnote{they are solutions to
$
  (1-y)^p \equiv 1 - y^p  \mod p^2,
$
viewed as elements of $\F_p\setminus\{0,1\}$.}.
Let $\rho$ be their number, and lift them to $r_1,...,r_\rho\in \Z$.
Pick one, say $r_i$, move it to zero with a shift $y\mapsto y+r_i$,
and apply Theorem \ref{mainthm2} to the shifted equation
$$
  (x-y+1-r_i)^p + (y+r_i)^p = 1. 
$$
The monomials defining its $\Dv$ are $x^p$, $-py^{p-1}$, $-pxy^{p-1}$,
$(r\!-\!1)^{p-1}px$, $c_2y^2$, and possibly $c_1y$, $c_0$. Here $c_0,c_1,c_2\in\Z$, and
$p||c_2$, $p^2|c_1$, $p^2|c_0$ by definition of a double root. 
The picture for $p=7$ is below (left):

\smallskip
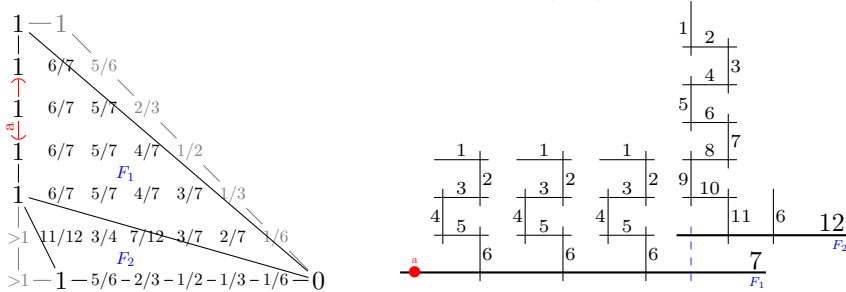
\quad\FERMATTWO
\begin{center}
Example: resolving one of the singularities for $p=7$
\end{center}

We see that $\Dv$ has two principal $v$-faces, $F_1$ from $-py^{p-1}$, $x^p$, $c_2y^2$
of denominator $p$, and $F_2$ from $(r\!-\!1)^{p-1}px$, $x^p$, $c_2y^2$ of denominator $2(p-1)$. 
Theorem \ref{mainthm2} now produces a model $\cC_i$, whose only singularities still come 
from the $v$-edge $L$. So we have resolved one of the singular points. 

Glueing the models $\cC_i$ together as in the first example above, we get a
\mrnc{} model $\cC/\Z_p$, which is also minimal regular.
The special fibre $\cC_k$ has the following components,
all $\P^1$s meeting transversally:
\begin{itemize}
\item 
Component $X_p$ of multiplicity $p$,
\item
$p-2\rho$ chains from $X_p$ with multiplicities $p-1,p-2,...,2,1$,
\item
Components $Y_1,...,Y_\rho$ of multiplicity $2p-2$, meeting $X_p$,
\item
A chain with multiplicities $2p-2,2p-1,...,2,1$ from each $Y_i$,
\item
A component of multiplicity $p-1$ from each $Y_i$.
\end{itemize}

When $p=3$, we have $\rho=0$ and $\cC_k$ is a type IV$^*$ elliptic curve, 
see Table~\ref{elltable}.

The multiplities on the chains come from the slopes. For example,
to see why $Y_i$ meets $X_p$ transversally, consider $\cC_i$ and compute the slopes at 
the $v$-edge $L'$ where $F_1$ meets $F_2$.
In the notation of \ref{defslopes}, let $P=(\frac{p-1}2,1)$.~Then 
$$
  s_1^{L'}=\tfrac{p}{2p-2},\qquad s_2^{L'}=\tfrac{(p+1)/2}{p},\qquad \smalldet p{(p+1)/2}{2p-2}{p}=1.
$$
The chain between $Y_i$ and $X_p$ is therefore empty; it continues to be empty 
in tame extensions of $\Z_p$ with $e|(2p-2)$.

\begin{remark}
\marginpar{Into intro?}
Theorem \ref{mainthm2} (and glueing as above) 
also gives an alternative approach to the construction of regular models 
\marginpar{And probably the semistability criterion}
for semistable hyperelliptic curves in odd residue characteristic \cite{M2D2}.
\end{remark}

%

%


\endsection
\section*{Appendix A. Elementary facts about lattice polygons in $\R^2$}
\let\oldthesection\thesection
\def\thesection{A}
\setcounter{equation}{0}


We start with equivalent characterisations of $\delta_F$ for 
a $v$-face $F$, and verify a statement that reflects rationality 
of \'etale cohomology (Thm. \ref{tamecoh}). 

\begin{lemma}
\label{lemdelta}
Let $F\subset\R^2$ be a convex lattice polygon,
\marginpar{lattice function?}
and $v: F\to\R$ an affine function, $\Z$-valued on the vertices of $F$. 
\begin{enumerate}
\item If $F$ has positive volume, then the following three numbers are equal:
\begin{itemize}
\item
$A = $ index of the lattice $v^{-1}(\Z)$ in $\Z^2$. 
\item 
$B = $ common denominator of $v(P)$ for $P\in\Z^2$. 
\item 
$C = $ common denominator of $v(P)$ for $P\in \bar F(\Z)$. 
\marginpar{Also, for every $v$-edge $L\subset \partial F$ and any $P\in\Z^2$ with $\LS(P)=1$,
we have $\delta_F=\lcm(\delta_L,\delta_P)$.}
\end{itemize}
\item
Fix $d\in\N_{\ge 2}$. Define two counting functions: for $n\in(\Z/d\Z)^\times$ let
$$
  N(\tfrac nd) = |\Delta(\Z)_{\frac nd+\Z}|, \qquad 
  \bar N(\tfrac nd) = |\bar\Delta(\Z)_{\frac nd+\Z}|.
$$
\begin{itemize}
\item[(i)]
$n\!\mapsto\! N(\frac nd)\!+\!N(\frac {-n}d)$, $n\!\mapsto\! \bar N(\frac nd)\!+\!\bar N(\frac {-n}d)$
are constant on $(\Z/d\Z)^\times$.
\item[(ii)]
If $F$ is a line segment or a parallelogram, 
then $N$, $\bar N$ are \hbox{constant}.
\end{itemize}
\end{enumerate}
\end{lemma}

\begin{proof}
(1)
$A\!=\!B$, $C|B$: clear from definitions.
$B|C$: Replace $F$ by any lattice triangle $T\subset F$. Now 
reflect $T$, completing it to a parallelogram, and tile the plane with it.

(2\emph{i}) 
\marginpar{see etalemults.m for a check}
Suppose $F$ is a line segment. 
Translate an endpoints of $F$ to $(0,0)$, and replace $v\mapsto v\!-\!v(0,0)$. This does not affect
the statement, but now $v: V(\Z)\to\Q$ is linear, where $V$ is the line generated by $F$. 
It descends to $v: \bar F(\Z)\setminus\text{\{endpoint\}}\to\Q/\Z$, and all reduced fractions $\frac nd$ 
have preimage sizes independent of $n$.
The same argument works for a parallelogram.

(2\emph{ii}) 
If $F$ is a triangle, reflect it in one of its edges $L$, completing it to 
a parallelogram $R=F\cup F'\cup L$. The values of $v$ on $F'$ are \emph{minus} those on $F$, 
so the claim follows from (2\emph{i}) for $L$ and for $R$.
In general, break $F$ into triangles and use (2\emph{i}) on the edges.
\end{proof}

\begin{lemma}
\label{lemphisum}
Let $\Delta\subset\R^2$ be a convex lattice polygon with positive volume, 
and $\cL=\{$1-dim faces of~$\Delta\}$.
For $L\in\cL$ let $\LS: \Z^2\surjects\Z$ be the unique
affine function vanishing on $L$ and non-negative on $\Delta$.
Then for every $(i,j)\in\Z^2$,
$$
  \sum_{L\in\cL} (|\bar L(\Z)|-1)\>(\LS(i,j)-1) \>\>=\>\> 2|\Delta(\Z)|-2.
  \eqno{(*_\Delta)}
$$
\end{lemma}

\begin{proof}
Induction on $|\bar\Delta(\Z)|$. If $|\bar\Delta(\Z)|$=3, then LHS$(*_\Delta)$=RHS$(*_\Delta)$=0. 
Otherwise cut $\Delta$ into two lattice polygons $A$, $B$; say $L_0=A\cap B$.
Then $(*_A)$, $(*_B)$ hold by induction and we claim that adding them up 
gives $(*_\Delta)$ minus $2(|L_0(\Z)|-1)$ in both sides.
For the right-hand side,
$$
  2|A(\Z)|\!-\!2 \>+\> 2|B(\Z)|\!-\!2 = 2(|\Delta(\Z)|-|L_0(\Z)|) -4.
$$
For the left-hand side, the terms of $(*_A)$, $(*_B)$ add up to those of $(*_\Delta)$
except the contribution from $L_0$. But $L_0^*$ with respect to $A$ is minus that for $B$,
so the two contributions also add up to $-2$ times $|L_0(\Z)|-1$.
%
%
\end{proof}

Finally, the proof of Baker's theorem \ref{thmbaker} used that
the differentials $\omega_{ij}$ behave well under $\Aut\G_m^2$:

\begin{lemma}
\label{diffGm}
Let $C: \{f(x,y)=0\}\subset\G_{m,K}^2$ be a smooth curve.
Consider a transformation of $\G_{m,K}^2$
$$
\begin{array}{llllll}
  x=\XX^a \YY^b && \XX=x^d y^{-b} \cr
  y=\XX^c \YY^d && \YY=x^{-c} y^a \cr
\end{array} 
\qquad\text{with}\qquad
\Bigl|\begin{matrix} a & b \cr c& d\end{matrix}\Bigr|=1.
$$
Let $F(\XX,\YY)=f(\XX^a \YY^b,\XX^c \YY^d)$, and assume that $f'_y$ and $F'_\YY$ are not identically zero. Then
$$
  \frac{1}{xy} \frac{dx}{f'_y} = \frac{1}{\XX\YY} \frac{d\XX}{F'_\YY}.
$$
\end{lemma}

\begin{proof}
From the chain rule for $f(x,y)=F(x^d y^{-b},x^{-c} y^a)$, we find
$$
\begin{pmatrix} xf'_x \cr yf'_y\cr \end{pmatrix} =
\begin{pmatrix} d&-c\cr -b&a\cr    \end{pmatrix}
\begin{pmatrix} \XX F'_\XX \cr \YY F'_\YY\cr \end{pmatrix}.
\eqno{(*)}
$$
On the other hand, using 
$\frac{(gh)'}{gh}=\frac{g'}{g}+\frac{h'}{h}$
and the relation $(dF=)F_\XX'd\XX+F_\YY'd\YY=0$ on $C$, we have
$$
\begin{pmatrix} \frac{dx}x \\[2pt] \frac{dy}y\cr \end{pmatrix} = 
\begin{pmatrix} a&b\cr c&d\cr    \end{pmatrix}
\begin{pmatrix} \frac{d\XX}\XX \\[2pt] \frac{d\YY}\YY\cr \end{pmatrix} = \frac 1{\XX\YY}
\begin{pmatrix} a&b\cr c&d\cr    \end{pmatrix}
\begin{pmatrix} \YY F'_\YY \cr \XX F'_\XX \cr \end{pmatrix} \frac{d\XX}{F'_\YY}.
\eqno{(**)}
$$
Combining $(*)$ (2nd row) and $(**)$ (first row), the formula follows.
\end{proof}

\let\thesection\oldthesection
\endsection
\section*{Appendix B. Inertia invariants on \'etale cohomology}
\let\oldthesection\thesection
\def\thesection{B}
\setcounter{equation}{0}

\marginpar{See AppB.jpg}

Let $O_K$ be a complete DVR, with field of fractions $K$ and perfect residue field $k$ of 
characteristic $p\ge 0$. 

\begin{theorem}
\label{bw2}
Let $\cC/O_K$ be a proper regular model of a smooth projective geometrically connected 
curve $C/K$. Write
\begin{itemize}
\item
$\cA=$ N\'eron model of $\Jac C$ over $O_K$,
\item
$\cC_{\bar k}, \cA_{\bar k}=$ special fibres of $\cC$ and $\cA$ base changed to $\bar k$,
\item 
$d=$ gcd(multiplicities of components of $\cC_{\bar k}$),
\item
\marginpar{Do we need all these subscripts? $[l^n]$ should force them automatically}
$\Phi=$ component group $\cA_{\bar k}/\cA^0_{\bar k}$.
\end{itemize}
For every prime $l\ne\vchar k$ and $n\ge 1$, there is a natural $G_k$-homomorphism
$$
  \Pic(\cC_{\bar k})[l^n] \lar \Pic(C_{\bar K})[l^n]^{I_K},
$$
with kernel killed by $d$ and cokernel killed by $d|\Phi|$. In the limit,
\begin{equation}\label{eqininv}
  V_l(\cA_{\bar k}^0) \>\>\iso\>\>
  V_l \Pic(\cC_k) \>\>\iso\>\> V_l\Pic(C_K)^{I_K}
\end{equation}
as $G_k$-modules; when $l\nmid d|\Phi|$, the same is true with $\Z_l$-coefficients.
\end{theorem}

\begin{proof}
This is essentially well-known, and the proof follows \cite[Prop 2.6]{BW} and \cite[\S 10.4.2]{Liu}
closely. (The claim is proved in \cite{BW} when $d=1$.)

Write $S_0=\Spec O_K$, $S=\Spec O_{K^{\nr}}$ and $X=\cC\times_{S_0}S$.
Denote by $X_s$ and $X_\eta$ the special fibre and the generic fibre of $X$, and 
$X_{\bar\eta}=X_\eta/K^s$.

Let $\Gamma_1,...,\Gamma_n$ be components of $X_s$, with multiplicities $d_1,...,d_n$, so that
$d=\gcd_i(d_i)$. Write $D$ for the free abelian group on the $D_i$ and $D^\vee$ for its dual,
with dual basis $\Gamma_1^*,...,\Gamma_n^*$. There are natural maps

\begin{tikzcd}[sep=6.8em]
  D \rar["\lambda","Z\mapsto \protect{[O_X(Z)]}"'] & \Pic X 
    \rar["\phi","\cL\mapsto \sum(\deg\cL|_{\Gamma_i})\Gamma_i^*"'] & D^\vee
\end{tikzcd}
\quad
\begin{tikzcd}[sep=6em]
  D \rar["\alpha=\phi\circ\lambda","Z\mapsto \sum(Z\cdot\Gamma_i)\Gamma_i^*"'] & D^\vee 
\end{tikzcd}

\smallskip
\noindent
with $\ker\alpha=d^{-1}X_s\Z$, and a commutative diagram with exact rows
\begin{equation}\label{commliu}
\begin{tikzcd}[sep=1.2em]
  \color{black}\ker s \dar[black,hook] \rar[black,hook] & \color{black}\lambda(D) \dar[black,hook] \rar[black,twoheadrightarrow] 
    & \dar[hook]\color{black}\alpha(D) \dar[black,hook] \\
  \ker\phi \dar[twoheadrightarrow]{s} \rar[hook] & \Pic X 
    \dar[twoheadrightarrow]\rar{\phi} & D^\vee\dar[twoheadrightarrow] \\
  \ker\psi \rar[hook] & \Pic X_\eta \rar{\psi} & \frac{D^\vee}{\alpha(D)}. \\
\end{tikzcd}
\end{equation}
Two bottom rows are as in \cite[Proof of 10.4.17]{Liu}, and the top row comes from 
the associated kernel-cokernel exact sequence, with zero cokernels.
As proved in \cite{Liu}, $\ker s=\ker(\phi: \lambda(D)\to\alpha(D))$ is killed by $d$
because $\ker\alpha=d^{-1}X_s\Z$ and $X_s\in\ker\lambda$. The multiplication by $l^n$
map applied to the left column of \eqref{commliu} gives a kernel-cokernel exact sequence
$$
  0 \lar \ker s[l^n] \lar \ker\phi[l^n] \lar \ker\psi[l^n] \lar \frac{\ker s}{l^n\ker s}.
$$
Now $\ker\phi[l^n]\iso \Pic X[l^n]$ from the middle row of \eqref{commliu}), 
as $D^\vee$ is torsion-free, and
$\Pic X[l^n]\iso \Pic X_s[l^n]$ by \cite[10/4.14-15]{Liu}. As for $\ker\psi$, 
from the bottom row of \eqref{commliu}
$$
  0 \lar \ker\psi[l^n] \lar \Pic X_\eta[l^n] \lar \frac{D^\vee}{\alpha(D)}[l^n] = \Phi[l^n],
$$
the last equality by \cite[10/4.12]{Liu}. Putting everything together, we get
\begin{equation}\label{piciso}
  \Pic X_s[l^n] \!\overarrow{\iso}\! \Pic X[l^n] \!\overarrow{\iso}\! \ker\phi[l^n] 
    \!\overarrow{(1)}\! \ker\psi[l^n]\>{\buildrel(2)\over\longinjects}\>\Pic X_\eta[l^n]
\end{equation}
with kernel and cokernel of (1) killed by $d$ and cokernel of (2) by $|\Phi|$.

Finally, $\Pic X_\eta=(\Pic X_{\bar\eta})^{I_K}$ from the exact sequence
(for any variety) 
$$
  H^1(G,K^s[X_{\bar\eta}]^\times) \lar \Pic X_\eta\lar (\Pic X_{\bar\eta})^{I_K} \lar 
  H^2(G,K^s[X_{\bar\eta}]^\times),
$$
see \cite[1.5.0, p.386]{CTS}. As $X_\eta$ is proper and geometrically integral,
$K^s[X_{\bar\eta}]^\times=(K^s)^\times$, so the first term is 0 by Hilbert 90, 
and the last one is the Brauer group, also 0 for a separably closed field.

This proves the first claim, and passing to the limit and tensoring with $\Q_l$ gives the other two;
also, $\cA_{\eta}[l^n]^{I_K}\iso \cA_s^0[l^n]$ by \cite[Lemma 2]{ST}.
\end{proof}

\expandafter\ifx\ACKN 1
\begin{acknowledgements}
I would like to thank Vladimir Dokchitser, Jack Thorne, Michael Harrison, Bartosz Naskrecki, 
Celine Maistret, Adam Morgan, David Holmes, Qing Liu, Raymond van Bommel, Natalia Afanassieva, 
Bjorn Poonen, Stefan Wewers, Christopher Doris, Michael Stoll, Dino Lorenzini,
Pip Goodman, Simone Muselli and the referee for their help with the paper.
\end{acknowledgements}
\fi

\let\thesection\oldthesection
\endsection

\sectionsstop

\renewcommand\baselinestretch{0.8}

\renewcommand\baselinestretch{1}

\marginpar{Arnold, Brieskorn, Laufer, Namikawa-Ueno [Ogg, Iitaka, Winters]}

\end{document}